\documentclass[10pt]{article}

\usepackage{latexsym,amsmath,amsthm,amssymb,eucal,stmaryrd}
\usepackage{times}
\usepackage{diagrams}
\def\URL#1{{\def\-{\discretionary{\kern 0pt}{}{}} #1}}
\newcommand{\R}{\mathbb{R}}

\newcommand{\N}{\mathbb{N}}

\newcommand{\oRp}{\overline{\mathbb{R}}_+}
\newcommand{\Rp}{{\mathbb{R}}_+}

\newcommand{\dbdownarrow}{\rlap{\raise.25ex\hbox{$\shortdownarrow$}}\raise-.25ex\hbox{$\shortdownarrow$}}
\newcommand{\dda}{\rlap{\raise.25ex\hbox{$\shortdownarrow$}}\raise-.25ex\hbox{$\shortdownarrow$}}
\newcommand{\dua}{\rlap{\raise-.25ex\hbox{$\shortuparrow$}}\raise.25ex\hbox{$\shortuparrow$}}
\newcommand{\dbuparrow}{\rlap{\raise-.25ex\hbox{$\shortuparrow$}}\raise.25ex\hbox{$\shortuparrow$}}

\renewcommand{\odot}{\cdot_{_{_{\mbox{}\!\!\!\!P}}}\kern-.2em}
\renewcommand{\boxdot}{\cdot_{_{_{\mbox{}\!\!\!\!H}}}\kern-.2em}
\renewcommand{\boxplus}{+_{_{_{\mbox{}\!\!\!H}}}\kern-.2em}

\newcommand{\cO}{\mathcal{O}}

\newcommand{\cRI}{\mathcal{RI}}
\newcommand{\llcurly}{\!\prec\!\!\prec}

\newcommand{\da}{\mathord{\downarrow}}
\newcommand{\ua}{\mathord{\uparrow}}

\newcommand{\norm}[1]{\|#1\|}

\newcommand{\dsup}{\mathop{\bigvee{}^{^{\,\makebox[0pt]{$\scriptstyle\uparrow$}}}}}

\newtheorem{satz}{Satz}[section]
\newtheorem{proposition}[satz]{Proposition}
\newtheorem{corollary}[satz]{Corollary}
\newtheorem{lemma}[satz]{Lemma}
\newtheorem{example}[satz]{Example}

\newtheorem{remark}[satz]{Remark}
\newtheorem{theorem}[satz]{Theorem}
\newtheorem{question}[satz]{Question}

\newtheorem{separation}[satz]{Separation Property}

\newtheorem{observation}[satz]{Observation}

\theoremstyle{margin}

\newcommand\LSC{\mathrm{LSC}}
\newcommand\MON{\mathrm{MON}}

\begin{document}
\title{%
The Cuntz semigroup and domain theory}

\author{Klaus Keimel}

\date{\today}

\maketitle

\vspace{1cm}

\begin{abstract}
Domain theory has its origins in Mathematics and Theoretical Computer
Science. Mathematically it combines order and topology. Its central
concepts have their origin
in the idea of approximating ideal objects by their relatively
finite or, more generally, relatively compact parts. 

The development of domain theory in recent years was mainly motivated
by question in denotational semantics and the theory of computation. But since
2008, domain theoretical notions and methods are used in the theory of
C$^*$-algebras in connection with the Cuntz semigroup.

This paper is largely expository. It presents those notions of domain
theory that seem to be relevant for the theory of Cuntz semigroups and
have sometimes been developed independently in both communities. It
also contains a new aspect in presenting results of Elliott,
Ivanescu and Santiago on the cone of traces of a C$^*$-algebra as a
particular case of the dual of a Cuntz semigroup.  
\end{abstract}

\section{Introduction}

Continuous lattices have emerged in quite distant areas under various
disguises, and the equivalence of the different definitions is not
straightforward. The two main sources are in topological algebra on
the one hand and in semantics of untyped $\lambda$-calculus at the other
hand.

In 1974, published in 1976 \cite{HS}, K.~H.~Hofmann and A.~R.~Stralka
arrived at the characterization that is now adopted generally.
In this work
on compact semilattices, it was their aim to characterize order
theoretically those compact Hausdorff semilattices that admit a separating
family of continuous semilattice homomorphism into the unit interval
$[0,1]$. These compact semilattices were also called \emph{Lawson
  semilattices}. The 
following relation turned out to be crucial (see 
\cite[p. 27, lines 20ff.]{HS}: For elements $x,y$ in a complete lattice they say
that $x$ is \emph{relatively compact in} $y$ if every open covering
$(u_i)_i$ of $y$ (aka $y\leq \sup_i u_i$) contains a finite subcover
$u_{i_1}, \dots,u_{i_n}$ of $x$ (aka $x\leq u_{i_1}\vee\dots\vee u_{i_n}$).   
This terminology was chosen since, in the lattice of open subsets of a
locally compact Hausdorff space, this relative compactness notion
agrees with the common use of relative compactness 
in topology: An open subset $V$ is relatively compact in an open
subset $W$ if the closure of $V$ is compact and contained in $W$. The
Lawson semilattices were characterized to be those
complete lattices, where each 
element is the supremum of its relatively compact parts, and they
called these lattices \emph{relatively algebraic}. Later on,
terminology changed: \emph{relatively compact in} was replaced by the
shorter \emph{way-below}.

Two years before, in 1972,  D.~S.~Scott's seminal paper \cite{Sc}
with the title \emph{Continuous Lattices} had appeared. In this paper Scott
provided the first models for the untyped $\lambda$-calculus using
what he had called \emph{continuous lattices}. It took some time until
the attention of the compact semilattice community was drawn towards
Scott's paper. It was only shortly before the appearance of \cite{HS}
in 1976 that it was discovered that Scott's continuous lattices were
precisely the relatively algebraic lattices in the sense of Hofmann
and Stralka.

Continuous lattices were mainly used in denotational semantics
of programming languages. In view of those applications a
generalization from complete lattices to directed complete partially
ordered sets (\emph{dcpos}, for short) was needed. Because of the lack of
finite suprema, the relation $x$ way-below $y$ had to be defined by
saying that every  directed family $(u_i)_i$ covering $y$ (aka $y\leq
\sup_i u_i$) contains an element $u_{i_o}$ covering $x$ (aka, $x\leq
u_{i_o}$), and a dcpo was said 
to be a \emph{continuous dcpo} (a \emph{domain}, for short), if each of its
elements $y$ is the supremum of a directed family of elements $x_i$
way-below $y$. The term 'domain' has its origin in the use
of these structures as semantic domains.

The author recently discovered that domain theoretic notions and
constructions are used in the theory of C$^*$-algebras. These
developments were initiated in a paper by Coward, Elliot and Ivanescu
\cite{CEI} in 2008. Their aim was to introduce a new invariant for
C$^*$-algebras that is finer than the K-groups. This invariant is
called the Cuntz semigroup and is a kind of completion of the
classical ordered semigroup introduced by J. Cuntz \cite{Cu78} in
1978. In \cite{CEI} and the follow-up papers
domain theory is not used in its classical form. A variant is
considered where the set system of directed subsets is replaced by
increasing sequences or, equivalently, by countable directed
sets. Thus, partially ordered sets are considered in which not all 
directed sets but only increasing sequences are required to have a
least upper bound. An element $x$ is said to be \emph{compactly
  contained in} $y$ if, 
every increasing sequence $u_n$ covering $y$ (aka $y\leq \sup_n u_n$)
contains an element $u_{n_0}$ already covering $x$ (aka $x\leq
u_{n_0}$). The Cuntz semigroup $S$ of a C$^*$-algebra as introduced in
\cite{CEI} has the following properties among others:

(O0) $S$ is a partially ordered commutative monoid with $0$ as
smallest element, 

(O1) every increasing sequence has a least upper bound, 

(O2) every element $y$ is the least upper bound of an increasing
sequence of elements $x_n$ compactly contained in $y$, 

(O3) if $x_i$ is compactly contained in $y_i$ for $i=1,2$, then
$x_1+x_2$ is compactly contained in $y_1+y_2$,       

(O4) addition is continuous in the sense that it preserves suprema of
increasing sequences.\\
A structure with these properties is then called an \emph{abstract
  Cuntz semigroup}.

A whole series of papers has appeared since that time with further
developments. The author of these lines has been working in domain
theory for more than 30 years. He discovered the new developments
around the Cuntz semigroups through a paper by Antoine, Perera and
Thiel \cite{APT}. It turns out that domain theoretical concepts and
methods play a more important role than expected. Quite some
properties have been rediscovered, other developments occur in 
parallel to developments in domain theory. 

This paper is largely expository. Its purpose is to establish a common
platform for communication between domain theory and the community
working on Cuntz semigroups. But it also pursues a specific purpose:
In 2011, Elliott, Robert and Santiago \cite{ERS} have published results on the
space of lower semicontinuous traces and
2-quasi-traces  of C$^*$-algebras. The proofs for the two cases seem
to follow a common pattern. The same pattern can be found in a paper
by Plotkin in 2009 \cite{Plo} on a
Banach-Alaoglu type theorems for continuous directed complete
partially ordered cones. Plotkin's results and methods
have been refined and generalized by the author just recently \cite{alaoglu}.
These results when specialized to abstract Cuntz semigroups give a
unified proof for the results of Elliott, Robert and Santiago. For
this, we show how the positive cone of a C$^*$-algebra can be viewed
as an abstract Cuntz semigroup. It is not amazing that the ingredients
for our proof can all be found in the paper of Elliott, Robert and
Santiago.  

In this presentation, I do not adopt the countable variant of domain
theory as used in the C$^*$-algebra community. I use dcpos, the
way-below relation and domains as in the monograph \cite{compend}. The
future will show, if I will be convinced to change to the countable
point of view. It is well-known that other subset systems 
can be used instead of directed sets, and quite analogous developments
can be carried through. One may consult a survey by Ern\'e \cite{E2}
on such variants of domain theory. 

The authors from the C$^*$-algebra community avoid the term
'way-below' as if it would be contagious. They use 'compactly contained
in', sometimes 'far below'. I do not mind other terminologies, but
remain with 'way-below' from time to time, and I hope that nobody
feels uneasy about it. \\

{\sc Acknowledgment} I am grateful to Hannes Thiel for corrections
and useful suggestions.

\section{Predomains and c-spaces}

We want to stress the concept of a predomain. 
In the same way as Hilbert spaces are completions of pre-Hilbert
spaces, domains are obtained from predomains by a completion process,
the round ideal completion. Domains can be defined in terms of
partial orders but have a strong topological 
flavor. Similarly, predomains occur under two different but
equivalent disguises: as relational and as topological structures. 

The notion of a predomain is not new at all. It is motivated by the 
notion of a basis for domains. 
This notion has been axiomatized as a relational structure first by
M. Smyth \cite{Sm} (under the name of an $R$-structure) and it occurs
under the name of an \emph{abstract basis} in standard texts on Domain
Theory, most prominently in the Handbook article by Abramsky and Jung
\cite[Section 2.2.6]{AJ}, where abstract bases are used for free 
constructions \cite[Chapter 6]{AJ}. This aspect has been rediscovered
by Antoine, Perera and Thiel \cite{APT} for constructing tensor
products of abstract Cuntz semigroups. The
topological variant is due to Ern\'e \cite{E,E1} under the name of a
c-space and independently to Ershov \cite {Er,Er1} under the name of
an $\alpha$-space. It was Ershov that insisted on omitting the
completeness properties required for domains. He had advocated this
aspect already in his early work on computable functionals of higher
type; his f-spaces and a-spaces are early manifestations (see
\cite{Er3,Er2}).  

It seems to me that these concepts have not yet
attracted the attention that they deserve. The defining properties are
amazingly simple 
and at the same time as powerful as those of domains. For this
reason, I propose a new name that stresses the importance by
calling them  \emph{predomains}.  

\subsection{Predomains}

Let us concentrate first on the relational aspect.
A \emph{predomain} is a set $P$ equipped with a binary relation $\llcurly$
that is transitive 
\[a\llcurly b\llcurly c \implies a \llcurly c \eqno{\rm (Trans)}\] 
and satisfies the following \emph{interpolation 
  property} for every finite subset $F$ and every element $c$: 
\[ F\llcurly c \implies  \exists b\in P.\   F\llcurly b\llcurly c
\eqno{\rm (IP)}\] 
where $F\llcurly c$ is an abbreviation for '$a\llcurly c$ for all $a\in
F$'.  

For $F$ we may choose the empty set and in this case the interpolation
property says:
\[\forall c.\ \exists b.\ b\llcurly c\eqno{\rm (IP0)}\]
Choosing $F$ to be a singleton, the interpolation property above
implies the ordinary interpolation property  
\[ a\llcurly c \implies \exists b.\ a\llcurly b\llcurly c\eqno{\rm (IP1)}\]
Choosing $F$ to be a two element set, the interpolation property
reads:
\[a_i\llcurly c\ (i=1,2) \implies \exists b.\ a_i\llcurly
b\llcurly c\ (i=1,2) \eqno{\rm (IP2)}\]
Clearly (IP0)and (IP2) together are equivalent to (IP).
We use the notation:
\[\dda c = \{b\in P\mid b\llcurly c\},\ \ \ \dua c = \{a\in P\mid
c\llcurly a\}\] 

The following is our basic example:
\begin{example}\label{ex:basic}
Let $X$ be a locally compact Hausdorff space, $C_0(X)$ the
$C^*$-algebra of all complex valued continuous functions defined on $X$ that
vanish at infinity. Its positive cone $C_0(X)_+$ consisting of those
$f\in C_0(X)$  with nonnegative real values is a poset with the usual
pointwise order $f\leq g$ if $f(x)\leq g(x)$ for all $x$. There is a
natural predomain structure on 
$C_0(X)_+$ defined by \[f\llcurly g \mbox{ if } f\leq (g-\varepsilon)_+\]
where $(g-\varepsilon)_+$ is the function with value $\max(g(x)-\varepsilon,0)$
for every $x\in X$.
\end{example}

For the relation $\llcurly$ on a predomain $P$ we use a terminology
borrowed from the partially ordered sets: A subset $D$ of $P$ is said
to be \emph{$\llcurly$-directed} if, for every finite subset $F$ of $D$,
there is an element $c\in D$ such that $F\llcurly c$. 

A subset $D'$ of a $\llcurly$-directed set $D$ is said to be 
\emph{$\llcurly$-cofinal} if, for every
$d\in D$ there is a $d'\in D'$ such that $d\llcurly
d'$. Such a $\llcurly$-cofinal subset $D'$ is also
$\llcurly$-directed. Indeed, for a finite subset $F\subseteq
D'\subseteq D$ there is an $d\in D$ such that $F\llcurly d$ and,
choosing an element $d'\in D'$ such that $d\llcurly d'$ we obtain
$F\llcurly d'$.  

A subset $Q$ of a predomain $P$ is said to be \emph{$\llcurly$-dense}
if, whenever $a\llcurly c$ holds for elements in $P$, there is an
element $b\in Q$ such that $a\llcurly b\llcurly c$. 

\begin{remark}\label{rem:dense}
A $\llcurly$-dense subset $Q$ of a predomain $P$ is a predomain when
equipped with the relation $\llcurly$ restricted to $Q$ and, for every
$c\in P$, the set $\dda_Q c = \dda c\cap Q$ is cofinal in $\dda c$.
\end{remark}

Clearly the restriction of $\llcurly$ to $Q$ is transitive. For the
interpolation property (IP) consider a finite subset $F$
of $Q$ and suppose $F\llcurly c$ for some $c\in Q$, then
$F\llcurly b \llcurly c$ for some $b\in P$ by (IP) and so we can find an
element $b'\in Q$ such that $b\llcurly b'\llcurly c$, whence
$F\llcurly b'\llcurly c$.

\subsection{Continuous posets and domains}
Let $(P,\leq)$ be a partially ordered set (\emph{poset}, for short). For
elements $a,b$ in $P$ we say that $a$ is relatively compact in $b$
($a$ is \emph{way-below} $b$, for short) and we write
$a\ll b$ if, for every directed subset $D$ such that $b\leq \sup D$,
there is an element $d\in D$ with $a\leq d$, whenever $D$ has a least
upper bound $\sup D$ in $P$. We say that $P$ is a \emph{continuous
  poset} if for every 
element $b\in P$ the set \[\dda b = \{a\in P\mid a\ll b\} \] is
directed and $b=\sup \dda b$. 

In a continuous poset, if $a\ll b$ and if $D$ is a directed subset
such that $b\leq \sup D$, there is a $d\in D$ such that even $a\ll d$.

If $(P,\leq)$ is a partially ordered set such that every directed
subset has a supremum, we say that $P$ is \emph{directed complete} (a
\emph{dcpo} for short). A continuous dcpo is called a \emph{domain}.

The relation $\ll$ in a poset $P$ has the following properties:
\begin{eqnarray}
a\ll c &\implies a\leq c\\
a\ll b\leq c &\implies a\ll c\\
d\leq a\ll b &\implies d\ll b
\end{eqnarray}

\begin{remark}
Every continuous poset is a predomain, when equipped with its
relation $\ll$.
\end{remark}

\begin{proof}
For transitivity, suppose that $a\ll b\ll c$. Then $b\leq c$ by
property (1), whence $a\ll c$ by property (2). For the interpolation
property (IP) let $F\ll a$ for a finite subset $F$. The family of sets
$\dda b$ with $b\ll a$ is directed, and each of the sets $\dda b$ is
directed. Thus $D = \bigcup_{b\ll a}\dda b$ is 
directed, too, and $\sup D = a$. For every $f\in F$ we  have $f\ll
a$. Thus, there is an element 
$d_f\in D$ with $f\leq d_f$. Since $D$ is directed, we find an element
$d\in D$ such that $f\leq d$ for every element $f$ in the finite set
$F$. Since for $d\in D$ there is an element $b$ such that
$d\ll b\ll a$, we have $F\ll b\ll a$.
\end{proof} 

Any $\ll$-dense subset $B$ of a continuous poset $P$ is called a
\emph{basis} of $P$. By remark \ref{rem:dense} every basis $B$ is a
predomain for 
the relation $\ll$ restricted to $B$; for every $c\in P$, the set
$\dda_{\!B} c = \dda c\cap B$ is directed and cofinal in $\dda c$ so that
$c=\sup \dda_{\!B} c$.

 But there are important predomain structures for which
relation $\llcurly$ is not derived from a partial order as above. This
is illustrated best by our basic example \ref{ex:basic} of the cone
$C_0(X)_+$ of nonnegative continuous real valued functions vanishing
at infinity on a locally compact Hausdorff space $X$. This cone
carries a natural pointwise order $f\leq g$ if $f(x)\leq g(x)$ for all
$x\in X$. The predomain relation $\llcurly$ does not agree with the
relation $\ll$ on   $C_0(X)_+$ derived from the partial order except
for very special cases. Let us choose $X$ to be the unit interval with
its usual compact Hausdorff topology where we denote by $\mathbf 1$ the
constant function with value $1$. Then $(1-\varepsilon)\mathbf 1
\llcurly \mathbf 1$, but  $(1-\varepsilon)\mathbf 1
\not\ll \mathbf 1$. Indeed, $f_n(x) =x^{1/n}$ is an increasing
sequence of continuous functions and $\mathbf 1$ is the least upper
bound of this sequence in the poset $(C_0([0,1])_+,\leq)$ (although not the
pointwise supremum) and $(1-\varepsilon)\cdot\mathbf 1)\not\leq f_n$ for
all $n$. Thus $(1-\varepsilon)\cdot\mathbf 1)\not\ll \mathbf 1$. By a
similar argument one can show that
there is no $f\ll\mathbf 1$ except for the constant function $0$.

The point in the example $(C_0(X)_+,\leq)$ is that there is a
difference between least upper bounds in the poset 
$(C_0(X)_+,\leq)$ and pointwise least upper bounds. We say that a function
$f\colon X\to\Rp$ vanishing at infinity is the pointwise supremum of
an increasing sequence or a directed 
family of functions $f_i$ in $C_0(X)_+$ if $f(x)=sup_if_i(x)$ for
every $x\in X$. By Dini's theorem, $f$ is then continuous and the
$f_i$ converge to $f$ uniformly. The predomain relation $\llcurly$ may
be defined by using this strengthened notion of pointwise least upper
bound instead of the notion of a least upper bound in the poset
$(C_0(X)_+,\leq)$. 

It is important to consider predomain structures $\llcurly$ not derived
from partial orders as $\ll$ in the case of continuous
posets. In the contrary, partial orders can be derived from predomain
structures as we will see. Predomains are more general and may be more important
than continuous posets. 

In the same vein, I propose to replace the notion of a
preCuntz semigroup as considered in \cite[Definition 2.1]{ABP} by a
more appropriate structure: commutative predomain monoids with an
additive relation $\llcurly$ (see below \ref{subsec:precuntz}).

\subsection{The round ideal completion}\label{subsec:2.3}

We have seen that every domain $D$ is a predomain for its way-below relation. 
More importantly, predomains occur as bases of domains. 
Let us see that every predomain has a completion which is a domain.

A \emph{round ideal} is a subset $J$ of a predomain $P$ with the
following properties: (1) $J$ is $\llcurly$-directed and (2) if
$a\llcurly b\in J$, then $a\in J$. This is equivalent to the
requirement that a finite
subset $F$ of $P$ is contained in $J$ if and only if there is an
element $b\in J$ such that $F\llcurly b$. 

For every element $b\in P$, the set 
\[\dda b = \{a\in P\mid a\llcurly b\}\] 
is a round ideal.

\begin{proposition}\label{prop:ricompletion}
The set $\cRI(P)$ of all round ideals of a predomain $P$ ordered by
inclusion is a domain, called the \emph{round ideal completion} of the
predomain $P$. 
The way-below relation on $\cRI(P)$ is given by $I\ll J$ if
there is an element $b\in J$ such that $I\subseteq \dda b$. The round
ideals $\dda a$, $a\in P$, form a basis of the round ideal
completion.
\end{proposition}

\begin{proof}
Since the union of a family of round ideals that is directed under
inclusion is a round ideal, the collection $\cRI(P)$ of all round
ideals is directed complete. 

Given two round ideals $I$ and $J$,
suppose that $I\ll J$. Since $J$ is the union of the round ideals
$\dda c$ with $c\in J$, we obtain $I\subseteq \dda c$ for some $c\in
J$. Suppose conversely that this latter condition is satisfied and
suppose that $J$ is contained in the union of a directed family of
round ideals $J_j$. Then $c\in J_i$ for some $i$ and consequently $\dda
c\subseteq J_i$. Hence $I\ll J$.   

By the characterization of the way-below relation, the
round ideals of the form $\dda c$, $c\in P$, are $\ll$-dense in
$\cRI(P)$ and, hence, form a basis.
\end{proof}

Since the round ideals $\dda c$, $c\in P$, form a basis for the
round ideal completion, they form a predomain, when equipped with the
restriction of the relation $\ll$ on $\cRI(P)$. One may conjecture that
$a\llcurly b$ if, and only if, $\dda a\ll \dda b$.
It is indeed true that $a\llcurly b$ in $P$ implies $\dda a\ll \dda b$
in $\cRI(P)$. But the converse is not true in general as the following
example shows (thus, not every predomain is the basis of a domain):

\begin{example}
Let $D$ be the union of $[0,1]^2$ and the segment $\{r(1,1) \mid 1\leq
r\leq 2\}$ in $\mathbb R^2$. On $D$ we take the coordinatewise
order. Then $D$ is a continuous lattice with the 
way-below relation: $(a,b)\ll(a',b')$ iff $1<a'=b', (a,b)\in[0,1]^2$
or $a<a'\leq 1, b<b'\leq 1$ or $a<a'\leq 1, b=b'=0$ or $a=a'=0,
b<b'\leq 1$ or $a=a'=b=b'=0$. 

We can weaken this way-below relation to a relation $\llcurly$ by
strengthening the first set of inequalities  to $1<a'=b', a<1$ or
$b<1$. Thus, for example $(1,a)\ll (2,2)$ but not $(1,a)\llcurly
(2,2)$. The round ideal completion of $(D,\llcurly)$ is the continuous
lattice $D$. 
\end{example} 

A predomain is called \emph{stratified} if  
\[\dda a\ll \dda b \mbox{ in } \cRI(P)
\implies a\llcurly b \mbox{ in } P\] 
By the characterization of the relation $\ll$ in Proposition
\ref{prop:ricompletion}, this is equivalent to 
\[\dda a\subseteq \dda c, c\llcurly b \implies a\llcurly b\]
Every predomain can be stratified by strengthening the relation
$\llcurly$ to: $a \llcurly_s b$ iff $\dda a\ll \dda b$ in the round
ideal completion iff there is a $c\llcurly b$ such that $\dda a
\subseteq \dda c$. 

Replacing $\llcurly$ by $\llcurly_s$ on a predomain $P$ does not change
the round ideal completion. Indeed, a domain is the round ideal
completion of any of its bases, and the predomain $(P,\llcurly_s)$ may
be identified with the basis of all $\dda a$, $a\in P$, of the
round ideal completion $\cRI(P)$. 

The containment order on round ideals induces a \emph{natural
  preorder} on the predomain $P$: Define $a\leq a'$ if $\dda a\subseteq
\dda a'$, that is, if $c\ll a$ implies $c\ll a'$. 

The property  
\[a\llcurly c \leq b \implies a\llcurly b \eqno{\rm (2)}\]
then holds for any predomain. The corresponding property
\[a\leq c \llcurly b \implies a\llcurly b \eqno{\rm (3)}\]
does not hold for predomains, in general; it holds if and only if the
predomain is stratified.

\begin{example}\label{ex:lsc}
Let us return to our basic example $C_0(X)_+$ of nonnegative real-valued
functions vanishing at infinity defined on a locally compact
Hausdorff space $X$ viewed as predomain as in \ref{ex:basic}. 
This predomain is stratified. Its round ideal completion can be
identified with the domain $\LSC(X)$ of all lower semicontinuous functions
$g$ from $X$ to the one point compactification $\oRp
=\Rp\cup\{+\infty\}$ of the nonnegative reals. A round ideal $J$ of
$C_0(X)_+$ is identified with the function $g$ defined by $g(x) =
\sup_{f\in J}f(x)$. The way-below relation
$\ll$ on $\LSC(X)$ is given by $g\ll h$ if there is an $f\in C_0(X)_+$
such that $g\leq f\leq (h-\varepsilon)_+$ for some $\varepsilon > 0$. 
\end{example} 


\subsection{c-Spaces}

Let us turn now to the topological variant of predomains. The
topologies occurring in this context are highly non-Hausdorff. This is
not a default but an essential feature. Indeed these topologies
combine order and topology.

In an arbitrary
topological space $(X,\tau)$ we use the \emph{specialization 
preorder}: $a\leq_\tau b$ if $a$ belongs to the closure of the singleton
$\{b\}$, which is equivalent to saying that every open neighborhood of
$a$ is also a neighborhood of 
$b$. For any element $x$ we denote by \[\ua x = \{y\in X\mid
x\leq_\tau y\}\] the 
\emph{saturation} of $x$, 
equivalently, the intersection of all open sets containing
$x$. Continuous functions between topological spaces preserve the
respective specialization preorders.

A \emph{c-space} is a topological space $X$ with the
property that every element $b$ has a neighborhood basis of sets of
the form $\ua x$. 
c-Spaces have the remarkable property that separate continuity is
equivalent to joint continuity, a property that has been noticed by
Ershov: 

\begin{proposition}\label{prop:joint}
Let $X$ be a c-space and $Y$ and $Z$ arbitrary topological spaces. Then
every map $f\colon X\times Y\to Z$ that is continuous separately in
each of the two arguments is jointly continuous.
\end{proposition}

\begin{proof}
Let $x_0$ and $y_0$ be elements of $X$ and $Y$, respectively, and $U$ a
neighborhood of $f(x_0,y_0)$. If $x\mapsto f(x,y_0)$ is continuous,
there is a neighborhood $W$ of $x_0$ such that $f(x,y_0)\in U$ for
every $x\in W$. Since $X$ is a c-space, we may suppose that $W = \ua
x_1$ for some $x_1\in X$. Using that $y\mapsto f(x_1,y)$ is continuous,
we find a neighborhood $V$ of $y_0$ such that $f(x_1,y)\in U$ for all
$y\in V$. Since $x\mapsto f(x,y)$ is continuous for every $y$, these
maps preserve the specialization order. Hence, $f(x,y)\in U$ for all
$(x,y)\in\ua x_1\times V$ and the latter set is a neighborhood of $(x_0,y_0)$. 
\end{proof}

\subsection{Predomains and c-spaces}

Every predomain $(P,\llcurly)$ carries a natural topology
$\tau_{\llcurly}$ that turns it into a c-space:   
A subset $U of P$ is declared to be open, if (1) $x\in U$ and $x\llcurly y$
imply $y\in U$ and (2) for every $x\in U$ there is an element $z\in U$
with $z\llcurly x$, 
equivalently, if  $U=\dua U$ where \[\dua
U =\{x\in P\mid \exists z\in U.\ z\llcurly x\}\] 

\begin{proposition}
On a predomain $(P,\llcurly)$ the open sets just defined form a c-space topology
denoted by $\tau_{\llcurly}$ for which the sets $\dua x =\{y\in
P\mid x\llcurly y\}, x\in P$ form a basis. The specialization preorder
agrees with the natural preorder of $(P,\llcurly)$.
\end{proposition} 

\begin{proof}
Clearly the union of any family of open sets in the sense just defined
is open. The intersection of finitely many open sets is open by the
interpolation property (IP). The sets of the form $\dua x$ are open by
transitivity and the interpolation property (IP1) and they form a basis for
the topology $\tau_{\llcurly}$ by the interpolation property
(IP2). The specialization preorder for this topology agrees with the
natural preorder since $x\leq_s y$ iff every open neighborhood of $x$
contains $y$ iff $z\llcurly x$ implies $z\llcurly y$ iff $\dda
x\subseteq \dda y$ iff $x\leq y$. Moreover every open neighborhood $U$
of $x$ contains an element $z\ll x$ so that $\ua z$ is a subset of $U$
containing the open basic neighborhood $\dua z$ of $x$. Thus
$(P,\tau_{\llcurly})$ is a c-space.
\end{proof}

Conversely, on a c-space $(P,\tau)$ we consider 
 the \emph{topological way-below relation}
$a\llcurly_\tau b$ if $\ua a$ is a neighborhood of $b$. 

\begin{proposition}
A c-space $(P,\tau)$ becomes a stratified predomain for the
topological way-below relation $\llcurly_\tau$. The natural preorder
associated with the relation $\llcurly_\tau$ agrees with the
specialization preorder $\leq_\tau$.
\end{proposition}

The topological way-below relation has the property that $a'\leq
a\llcurly_\tau b\implies a'\llcurly_\tau b$, that is, it satisfies
property (3). Thus, $(P,\llcurly_\tau)$ is a stratified predomain.
 
The two constructions almost yield a one-to-one correspondence between
predomains and c-spaces. Starting with a predomain $(P,\llcurly)$,
then passing to the c-space 
topology $\tau_{\llcurly}$ and then extracting the topological
way-below relation yields the stratification of the original relation
$\llcurly$. Starting with a c-space, extracting its topological
way-below relation and forming then the associated c-space topology
gives back the original c-space topology.   

On a domain, the c-space topology agrees with the Scott
topology. Indeed, on a domain the sets of the form $\dua a$ form a
basis of the Scott topology \cite[Theorem II-1.14]{compend}. 

%

\subsection{Countability conditions}

A predomain $P$ is said to be \emph{first countable} if, for every
element $b$, the round ideal $\dda b$ has a countable $\llcurly$-cofinal
subset. This is equivalent to the requirement that 
there is a sequence $a_1\llcurly a_2\llcurly\dots$ which is
$\llcurly$-cofinal in $\dda b$. We say that
$P$ is \emph{second countable} or \emph{countably based} if it has a
countable $\llcurly$-dense subset $B$. 

Our basic example $(C_0(X)_+,\llcurly)$ \ref{ex:basic} is first
countable choosing $f_n=(f-\frac{1}{n})_+$, $n\in\N$, but not second
countable in general.   

The first and second countability conditions for predomains correspond
to first and second countability of the corresponding c-space topology
$\tau_{\llcurly}$, respectively. 

The round ideal completion of a countably based predomain is countably
based. But the round ideal completion of a first countable predomain
need not be first countable: 

\begin{example}\label{ex:basic3}
Our basic example $(C_0(X)_+,\llcurly)$ \ref{ex:basic} is first
countable. Choosing $f_n=(f-\frac{1}{n})_+$, $n\in\N$, we obtain a
countable cofinal subset in the round ideal $\dda f$. The round ideal
completion need not be first countable. As an example, let $X$ be an
uncountable discrete space. The round ideal completion of $C_0(X)_+$
is the domain of 
all maps $g$ from $X$ to $\oRp$. The maps $f\ll g$ are those maps $f$ with
finite support that satisfy $f(x)<g(x)$ for all $x$ in the support of
$f$. Because of the uncountablity of $X$, there cannot be a cofinal
countable subset among the functions $f$ way-below, for example, the
constant function 1.  
\end{example}

A round ideal $I$ will be said to be countably generated or simply a round
$\omega$-ideal if it contains a
sequence \[a_1\llcurly a_2\ll \dots\] such that, for every $b\in I$, there is
an $n$ such that $b\ll a_n$. If $P$ is first countable, then $i(a) =
\dda a$ is a round $\omega$-ideal and we may form the $\omega$-completion,
the collection $\omega\cRI(P)$  of all round $\omega$-ideals
which is $\omega$-complete in the sense that the union of every
increasing sequence of round $\omega$-deals is a round $\omega$-ideal. 
The round $\omega$-ideal completion is an \emph{$\omega$-domain}. By this we
mean that every element $a$ is the supremum of a chain $a_1\ll_\omega
a_2\ll_\omega\dots$, where $b\ll_\omega a$ if, for every sequence
$b_1\leq b_2\leq \dots$ with $a\leq \sup_n b_n$, there is an $n$ such
that $b\leq b_n$.  In a first countable 
predomain there may exist round ideals that are not countably
generated as we have seen in \ref{ex:basic3}. Another example is given by
the ordered set $\Omega$ of 
countable ordinals. Here, the set $\Omega$ itself is a round ideal
that is not countably generated. 

In the literature related to the Cuntz semigroup, first countability
is always required following Coward, Elliott and Santiago \cite{CEI}.

\subsection{Morphisms}\label{sec:mor}
For predomains it is natural to consider maps
$f\colon P\to Q$ that \emph{preserve the relation $\llcurly$}, that is
$a\llcurly b$ implies $f(a)\llcurly f(b)$. For the
associated c-space topology that is equivalent to saying that $\ua f(U)$ is
open for every open subset $U$. Maps between topological
spaces will be called \emph{open}, if they satisfy this property. (In
topology, a map 
is called open if the image of every open set 
is open. We have modified this definition, but in such a way that for
T$_1$ spaces the new definition agrees with the old one. For
T$_0$-spaces this new definition looks more appropriate.)  

As for topological spaces in general, for c-spaces it is natural to
consider continuous maps. Continuous maps preserve the respective
specialization preorders, but not the topological way-below relations.
Accordingly, a map $f\colon P\to Q$ between predomains will be called
\emph{continuous}, if it is continuous for the respective c-space topologies.
This is equivalent to the condition:
\[\forall b\in P.\ \forall c\in Q.\ c\llcurly f(b) \implies \exists
a\in P.\ a\llcurly b \mbox{ and } c\llcurly f(a)\] 

The canonical map $i\colon a\mapsto \dda a$ from a predomain $P$ into its round
ideal completion preserves the relation $\llcurly$. It also is
continuous. Indeed, if $I$ is a round ideal with 
$I\ll \dda a$, then there is an element $b\in \dda a$ such that
$I\subseteq \dda b$. As a consequence, $i\colon a\mapsto \dda a$ is a
topological embedding.

On posets and dcpos we use Scott continuity. A map between posets is
said to be \emph{Scott continuous} if it is monotone and preserves
existing suprema of directed sets\footnote{It is interesting to remark
that Hofmann and Stralka in their 1976 paper \cite[Definition
1.29]{HS} had proposed to 
call \emph{normal} those maps that preserve existing directed suprema
in analogy to the terminology used for W$^*$-algebras}. This order
theoretic notion of   
continuity is equivalent to continuity with respect to a topology, the
Scott topology. The closed sets of the \emph{Scott topology} of a
poset are those lower sets that are closed for suprema of directed
subsets, as far as they exist. 

On domains one can consider the Scott topology and the associated
c-space topology. Fortunately the two topologies agree so that there
is no ambiguity when talking about continuity of function from or into domains.

\subsection{Universality of the round ideal completion}

The canonical map $i\colon a\mapsto \dda a$ from a predomain
$(P,\llcurly)$ into its round 
ideal completion is continuous and preserves the relation
$\llcurly$. Both properties are consequences of the characterization of the
way-below relation on the round ideal completion: Indeed, if $I$ is a
round ideal with $I\ll \dda a$, then there is an element $c\in \dda a$ such that
$I\subseteq \dda c$ which shows continuity. For the preservation of
$\llcurly$, let  
$a\llcurly b$; interpolate an element $a\llcurly c\llcurly b$ and we 
have $\dda a\subseteq \dda c$ and $c\in \dda b$, that is $\dda
a\ll \dda b$.

The round ideal completion of a predomain
has the desired universal property: 

\begin{proposition}\label{prop:universal}
For every continuous map $f$
from a predomain $(P,\llcurly)$ into a dcpo $Q$ (with the Scott
topology), there is a unique continuous map 
$\widehat f\colon \cRI(P)\to Q$ such that $\widehat f(\dda a) = f(a)$: 
\begin{diagram}
P&\rTo^\dda      & \cRI(P)\\
   &\rdTo_f&\dTo>{\widehat f}\\
   &            & Q
\end{diagram}
If $Q$ is a domain, the continuous extension $\widehat f$ preserves $\ll$ if
and only if $f$ preserves $\ll$.
\end{proposition}

\begin{proof}
For uniqueness suppose that $\widehat f\colon \cRI(P)\to Q$ is a
continuous map satisfying $\widehat f(\dda a) = f(a)$. Any round
ideal $J$ is the union of the directed family of $\dda a$, $a\in
J$. Thus $\widehat f(J) = \widehat f(\bigcup_{a\in P}\dda a)
=\sup_{a\in P}\widehat f(\dda a)\mbox{ (by continuity) } = \sup_{a\in
  P}f(a)$.

We now define $\widehat f$ by $\widehat f(J) =\sup_{a\in J}f(a)$ for
every round ideal $J$. We first remark that $\widehat f$ is well
defined, since for a round ideal $J$ of $P$ the image $f(J)$ is
directed. Indeed, a 
continuous map preserves the specialization order; the specialization
order on a predomain is the order $a\leq b$ iff $\dda\subseteq \dda b$
and the specialization order for the Scott topology on a dcpo is the
given order. 
 
We now check that $\widehat f$ is Scott-continuous. If $J_i$ is a
directed family of round ideals the $J=\bigcup_i J_i$ is its supremum
in the domain of round ideals and $\widehat f(J) = \sup_{a\in J} f(a)
= \sup_i \sup_{a\in J_i}f(a) =  \sup_i\widehat f(J_i)$. 

Now suppose that $Q$ is a continuous dcpo. If $\widehat f$ preserves
$\ll$, then $f=\widehat f\circ\dda$ preserves $\ll$. Conversely,
suppose that $f$ preserves $\ll$. Let $I\ll J$ in $\cRI(A)$. There is
an $a\in J$ such that $b\ll a$ for all $b\in I$. Let $a\ll a'\in
J$. Then $\widehat f(I) =\sup_{b\in I} f(b) \leq f(a) \ll f(a') \leq
\dsup_{b\in J}f(b) = \widehat f(J)$. Thus, $\widehat f$ preserves $\ll$.
\end{proof}  

For a continuous map $f$ from a predomain $P$ to a predomain $Q$, the
composition $\dda\circ f\colon P\to \cRI(Q)$ is continuous, too. By
the preceding proposition, there is a unique continuous map $\cRI(f)\colon
\cRI(P)\to \cRI(Q)$ such that $\cRI(f)\circ\dda = \dda\circ f$:   
\begin{diagram}
P&\rTo^\dda &\cRI(P)\\
\dTo< f&&\dTo>{\cRI(f)}\\
Q&\rTo_\dda&\cRI(Q)
\end{diagram}
and this map is defined by  $\cRI(f)(J) = \bigcup_{a\in J}\dda f(a)=
\dda f(J)$. Moreover, $\cRI(f)$ preserves $\ll$ if and only if $f$
does. In this way, $\cRI$ becomes a functor from the category of
predomains and continuous maps to the category of domains and
continuous maps. It restricts to a functor if one restricts to
continuous maps preserving $\ll$.

From a topological point of view, the round ideal completion 
$\cRI(X)$ of a c-space $X$ can also 
be seen to be the D-completion in the sense of \cite[Proposition
9.1]{KL} and equivalently as the sobrification \cite[Proposition
10.2]{KL} . Thus it has a more general universal property than the one
shown above: For every 
continuous map $f$ from $X$ into a monotone convergence space $Y$ (in
particular, into every sober space),
there is a unique continuous map $\widehat f\colon \cRI(X)\to Y$ such
that $\widehat f(\dda x) = f(x)$ for all $x\in X$ \cite[Theorem 6.7]{KL}.
But we will not use this more general point of view in this paper.



Not all continuous maps from $\cRI(P)$ to $\cRI(Q)$ are induced by
continuous maps from $P$ to $Q$, but only those that map the
basis $P$ to the basis $Q$. As the continuous maps from  $\cRI(P)$ to
$\cRI(Q)$ are in one-to-one correspondence with the continuous maps
from $P$ to $\cRI(Q)$, we may view these maps $F$ as set-valued maps from
$P$ to $Q$, where $F(x)$ is a round ideal of $Q$ for every $x\in
P$. Alternatively we may view these maps as a relation $R\subseteq P\times Q$
where $(x,y)\in R$ if $y\in F(x)$. It is not difficult to axiomatize
such relations; one has to write down, firstly, that the set of all $y$ such
that $(x,y)\in R$ form a round ideal and, secondly, that $F$ is continuous.
One finds such axioms in \cite[Definition 2.2.27]{AJ}.

\section{PreCuntz semigroups and their duals}

In this paper, a \emph{monoid} $(C,+,0)$ will always be understood to be
commutative. Thus, $+$ is a commutative associative operation with
neutral element $0$. A \emph{monoid homomorphism} is a map $f$ between monoids
such that $f(0)=0$ and $f(x+y)=f(x)+f(y)$. 

A \emph{cone} is a monoid $(C,+,0)$ endowed with a
scalar multiplication by real numbers $r>0$ which  
satisfies the identities: $1\cdot a = a, (rs)a =r(sa),
(r+s)a=ra+sa, r(a+b) =ra+rb$. One may extend the scalar multiplication
to $r=0$ by defining $0\cdot a =0$, and the above laws for the scalar
multiplication remain valid. A \emph{linear} map is a map $f$ between
cones which is a monoid homomorphism and satisfies $f(rx) =rf(x)$.

\subsection{PreCuntz and Cuntz semigroups}\label{subsec:precuntz}

A \emph{predomain monoid} $(C,+,0,\llcurly)$ is a monoid
endowed with the structure 
of a predomain in such a way that addition is
continuous. 

We will say 
that the relation $\llcurly$ is \emph{additive} if \[0\llcurly a \ \
\mbox{ and }\ \ a\llcurly a', b\llcurly b'\implies a+b\llcurly a'+b'\]
A predomain monoid in which the relation $\llcurly$ is additive will
also be called a \emph{preCuntz semigroup}.

Since for c-spaces separate continuity implies joint continuity by Proposition
\ref{prop:joint}, it suffices to require addition to be separately
continuous, that is, the maps $x\mapsto a+x$ to be
continuous for every $a$. But the additivity of the relation
$\llcurly$ has to be understood jointly as defined
above. It is not sufficient to require that $a\llcurly a'$ implies
$a+b\llcurly a'+b$.

A directed complete partially ordered monoid (a dcpo-monoid, for
short) is a monoid with a directed complete partial order such that
the addition is (Scott-) continuous. If the underlying dcpo is a
domain, we say that it is a \emph{domain monoid}.  A domain monoid with
an additive way-below relation will be called a \emph{Cuntz
  semigroup}. Of course, a Cuntz semigroup is also preCuntz; we just
have to concentrate at the way-below relation. 

Let us look at the round ideal completion of a preCuntz semigroup
$C$. For two round ideals $I$ and $J$ define
\[ I+J = \bigcup_{a\in I, b\in J} \dda (a+b)\]
Then $I+J$ is a round ideal. Indeed, if $c\in I+J$, then $c\ll a+b$
for some $a\in I, b\in J$. There are $a'\in I$ and $b'\in J$ with
$a\llcurly a'$ and $b\llcurly b'$. By the additivity of the way-below
relation we obtain $c\llcurly a+b\llcurly a'+b'$, whence $c\in
\dda(a'+b')$, that is, $c\in I+J$. Using Proposition
\ref{prop:ricompletion} the following is easily verified:

\begin{proposition}
The round ideal completion $\cRI(C)$ of a preCuntz semigroup $C$ is a Cuntz
semigroup. The map $a\mapsto \dda a\colon C\to\cRI(C)$ is a continuous
monoid homomorphism preserving $\ll$. 
\end{proposition}

\begin{proof}
It just remains to verify that the continuous extension $\widehat f$
according to  
\ref{prop:universal} is a monoid homomorphism: $\widehat f(I+J) = \widehat 
f(\bigcup_{a\in I,b\in J}\dda (a+b) = \sup_{a\in I,b\in J} \widehat
f(\dda(a+b)) = \sup_{a\in I,b\in J}f(a+b) =  \sup_{a\in
  I,b\in J}f(a)+f(b) = \sup_{a\in I}f(a) + \sup_{b\in J}f(b) =
\sup_{a\in I}\widehat f(\dda a) + \sup_{b\in J}\widehat f(\dda b) =
\widehat f(I) +\widehat f(J)$. 
\end{proof}

The round ideal completion of a preCuntz semigroup we has the expected
universal property:

\begin{proposition}\label{prop:universal1}
If $f\colon C\to D$ is a 
continuous monoid homomorphism from
a preCuntz semigroup $C$ into a dcpo monoid $D$, the
unique continuous extension $\widehat f\colon \cRI(C)\to D$ satisfying
$\widehat f(\dda a) =f(a)$ for all $a\in C$ according to the universal
property \ref{prop:universal} is a monoid homomorphism.
\end{proposition}

\begin{proof}
It just remains to verify that the continuous extension $\widehat f$
according to  
\ref{prop:universal} is a monoid homomorphism: $\widehat f(I+J) = \widehat 
f(\bigcup_{a\in I,b\in J}\dda (a+b) = \sup_{a\in I,b\in J} \widehat
f(\dda(a+b)) = \sup_{a\in I,b\in J}f(a+b) =  \sup_{a\in
  I,b\in J}f(a)+f(b) = \sup_{a\in I}f(a) + \sup_{b\in J}f(b) =
\sup_{a\in I}\widehat f(\dda a) + \sup_{b\in J}\widehat f(\dda b) =
\widehat f(I) +\widehat f(J)$. 
\end{proof}

\begin{corollary}
For every continuous
monoid homomorphism $f\colon C \to D$ of preCuntz semigroups there is
a unique continuous monoid homomorphism $\cRI(f)\colon \cRI(C)\to
\cRI(D)$ such that $\cRI(f)(\dda a) = \dda f(a)$ for all $a\in C$, and
$\cRI(f)$ preserves $\ll$ if and only if $f$ does.  
\end{corollary}

We now add scalar multiplication, thus passing from monoids to
cones. The properties stated before remain valid.  
A \emph{preCuntz cone} is a preCuntz monoid which is also a cone such
that scalar multiplication is continuous as a map $\R_{>0}\times C\to
C$, where $\R_{>0}$ is considered as a predomain
with $<$ as its approximation relation. In a preCuntz cone the
map $x\mapsto rx\colon C\to C$ is continuous and it has a continuous
inverse $x\mapsto r^{-1}x$, that is, it is a homeomorphism. It follows 
 that $x\llcurly y$ implies $rx\llcurly ry$ for $0<r<+\infty$. 

The round ideal completion of a preCuntz cone is a Cuntz cone with a
universal property analogous to the universal property of the round
ideal completion of a preCuntz semigroup: A dcpo-cone is understood to
be a dcpo-monoid $D$ which is also a cone in such a way that the
scalar multiplication $\R_{>0}\times D\to D$ is Scott-continuous. For
every linear map $f$ from a preCuntz cone $C$ to a dcpo-cone $D$ the
unique continuous extension $\widehat f\colon \cRI(C)\to D$ is linear.    
 
\begin{example}
The nonnegative real numbers form a preCuntz cone $\Rp$ and $\oRp$
is its round ideal completion. In both cases the approximation relation is
$r\llcurly s$ if $r<s$ or $r=s=0$. The addition is the usual one,
extended by $r+\infty =+\infty$.  
\end{example}

\begin{example}
Our basic example $C_0(X)_+$ for a locally compact Hausdorff space
$X$ is a preCuntz cone with the usual pointwise defined addition
of functions.  But notice that $f\llcurly f'$ does
not imply $f+g\llcurly f'+g$. For example $(x-\frac{1}{2})_+\llcurly
(x-\frac{1}{4})_+$, but $x+ (x-\frac{1}{2})_+\not\llcurly
x+(x-\frac{1}{4})_+$. Also $r\ll s$ does not imply $rf\ll sf$ (for 
example if $X =[0,1]$ and $f(x)= x$, then $rf \not\ll sf$, whenever
$0<r<s$). 

The round ideal completion of $(C_0(X)_+,\ll)$ is $\LSC(X)_+$, the set
of all lower semicontinuous maps $f\colon X\to\oRp$. Here the way
below relation is given by $g\ll h$ if there is a $f\in C_0(X)_+$ and
an $\varepsilon >0$ such that $g\leq (f-\varepsilon)_+, f\leq h$. 

As a predomain, $C_0(X)$ is first countable, since
$(g-\frac{1}{n})_+ \ll (g-\frac{1}{n+1})_+$ and since for every $f\ll
g$ we have $f\ll (g-\frac{1}{n})_+$ for some $n$. I do
not think that the round $\omega$-ideal completion is what one wants
to consider here, except for those cases where it agrees with
$\LSC(X)_+$. 
\end{example}

\begin{remark}{\rm 
The notion of an abstract Cuntz semigroup has been introduced by
Coward, Elliott and Ivanescu \cite{CEI}.  
First countable preCuntz semigroups have been introduced by Antoine,
Perera and Thiel under the name of a pre-W-semigroup \cite[Section
2.1]{APT}. They  construct their $\omega$-round ideal completions and
prove a universal property of this construction \cite[Chapter
3]{APT}. }
\end{remark}

\subsection{Topologies on posets and function spaces}
\label{subsec:3.2}
Let $L$ be a poset, the order relation being denoted by $\leq$. Denote
by $L^{op}$ the same set with the opposite order $\geq$. Besides the
Scott topology $\sigma$ on $L$ one may consider the dual Scott
topology $\sigma^{op}$, the Scott topology of $L^{op}$. We are
interested in three other topologies on $L$ that look quite simple at
a first glance.
\begin{verse}
The \emph{upper topology} $\tau_{up}$ has the principal ideals $\da b =
\{y\in L\mid y\leq b\}$, $b\in L$, as a subbasis for the closed sets.

The \emph{lower topology} $\tau_{lo}$ has the principal filters $\ua a =
\{y\in L\mid y\geq a\}$, $a\in L$, as a subbasis for the closed sets

The \emph{interval topology} $\tau_{iv}$ is generated by the upper and
the lower topology. The closed
intervals $[a,b] = \ua a \cap \da b$ are closed
sets.  
\end{verse}
We will use the following general observation: For complete lattices
$L$ and $M$, a map $\beta\colon L\to M$ preserving arbitrary suprema
has a lower adjoint $\alpha\colon M\to L$ defined by $\alpha(y)=\sup
\{x\in L\mid f(x)\leq y\}$. Then $\alpha$ preserves arbitrary meets
and $\beta^{-1}(\da y)=\da\alpha(y)$ which shows that $\beta$ is
continuous for the respective upper topologies. Similarly, $\alpha$ is
continuous for the respective upper topologies.  

On $\oRp$, the extended nonnegative reals, the proper open sets of the
upper topology are the intervals $]r,+\infty]$, the proper open sets
for the lower topology are the intervals $[0,r[$, and the interval
topology is the usual compact Hausdorff topology with the open
intervals $]r,s[$ as a basis for the open sets. The analogous
statement holds for subsets of $\oRp$ as the set of nonnegative reals
$\Rp$ and the set $\R_{>0}$ of positive reals. 

In agreement with classical analysis, a function from a space $X$ into
$\oRp$ is lower  
semicontinuous\footnote{It looks incoherent to call a function \emph{lower}
  semicontinuous if it is continuous for the \emph{upper} topology on
  $\oRp$. But this is unavoidable if one want to stay coherent with
  the use of lower semicontinuity in analysis and the terminology for
  topologies used in \cite{}.} if and only if it is continuous with
respect to the upper topology on $\oRp$.  We are interested in
a special case.

Let $(P,\llcurly)$ be a predomain with its c-space topology and its
natural preorder. We denote by $\LSC(P)$ the set of all lower
semicontinuous functions $f\colon P\to \oRp$ ordered pointwise: $f\leq
g$ if $f(x)\leq g(x)$ for all $x\in P$. We want to look at the
intrinsic upper, lower and interval topology on this function space. 

\begin{proposition}\label{prop:funcspaces}
For a predomain $P$ the function space $\LSC(P)$ has the following
properties: 

(a) A subbasis for the upper topology of $\LSC(P)$ is given by:  
\[V_{x,r} = \{f\in \LSC(P)\mid f(x) > r\},\ x\in P, r\in\Rp\]
A net $(f_i)_i$ of functions in $\LSC(P)$ converges to $f\in \LSC(P)$  
for the upper topology if and only if: 
\[f(x) \leq \liminf_i f_i(x) \mbox{ for all } x\in P \eqno{\rm (upConv)}\]  

(b)  A subbasis for the lower topology is given by: 
\[W_{y,r} = \{f\in \LSC(P)\mid f(x) < r \mbox{ for some } x\in \dua
y\},\ y\in P, r\in \Rp.\]
A net $(f_i)_i$ of functions in $\LSC(P)$ converges to $f\in \LSC(P)$ 
for the lower topology if and only if: 
\[\limsup_i f_i(y) \leq f(x)\ \mbox{ whenever } y\llcurly x\mbox{ in } P
\eqno{\rm (loConv)}\] 

(c) For the interval topology and the pointwise order, $\LSC(P)$ is
a compact ordered space. A net $(f_i)_i$ of functions in $\LSC(P)$
converges to $f\in \LSC(P)$ for the interval topology if and only
conditions {\rm (upConv)} and {\rm (upConv)} hold.  
\end{proposition}

The following developments contain a proof for the proposition.

We begin with a set $P$ and we consider the power $\oRp^P$ of all
functions $g\colon 
P\to\oRp$. With respect to the pointwise order,  $\oRp^P$ is a complete
lattice. Suprema and infima of arbitrary families of functions are
formed pointwise. The lower, upper and interval topology on $\oRp^P$
agree with the product topologies of the  lower, upper and interval
topology on $\oRp$, respectively. A net $(f_i)_i$ in $\oRp^P$ converges
to $f$ for the upper (resp., lower) topology if and only if does so
pointwise, that is, if and only if for every $x\in P$,
\[f(x) \leq\liminf_i f_i(x) \mbox{ (resp., } \limsup_i f_i(x)\leq
f(x) \mbox{)} \eqno{\rm (Conv)}\]   
As a product of compact ordered spaces, $\oRp^P$ is a compact ordered
space for the interval topology.

Suppose now that $P$ is a preordered set and consider the collection
$\MON(P)\subseteq \oRp^P$ of all monotone functions. Since
pointwise suprema and infima of monotone functions are monotone,
$\MON(P)$ is a complete sublattice of $\oRp^P$. Its 
intrinsic lower, upper and interval topology agree with the subspace topology
induced by the lower, upper and interval topology on
$\oRp^P$. Convergence is characterized as above and $\MON(P)$ is closed
in $\oRp^P$, hence, a compact ordered space, for the interval topology,

We now specialize further and suppose that $P$ is a topological space with
its specialization preorder. The lower semicontinuous functions
$f\colon P\to 
\oRp$ form a subset $\LSC(P)$ of $\MON(P)$, since continuous functions
preserve the specialization preorder. The pointwise supremum of a
family of lower semicontinuous functions is again lower
semicontinuous, that is, the canonical injection of $\LSC(P)$ into
$\MON(P)$ preserves arbitrary suprema. It follows that $\LSC(P)$ is a
complete lattice, too, and that the lower adjoint
$env\colon \MON(P)\to \LSC(P)$ that assigns to every order preserving
map $g\colon P\to \oRp$ its lower semicontinuous envelope $env(g) =
\sup \{f\in \LSC(P)\mid f\leq g\}$ preserves arbitrary infima. 
The lower semicontinuous envelope $env(g)$ is also given by
\[env(g)(x) = \liminf_{\mathfrak u_x}f(x) = \sup_{U\in\mathfrak
  u_x}\inf_{z\in U}f(z)  \eqno{\rm (Env)}\]
for every $x\in P$, where $\mathfrak u_x$ is any neighborhood basis of
$x$. The intrinsic upper topology of the lattice $\LSC(P)$ is the
subspace topology induced by the upper topology on $\oRp^P$. Indeed,
if $g\in \MON(P)$ and $f\in \LSC(P)$, then: \[f\leq g \mbox{ if and only if }
f\leq env(g) \eqno{\rm (Adj1)}\] Thus convergence in $\LSC(P)$ with
respect to the upper topology is characterized by condition
(upConv).  This proves claim (a). 

Infima in $\LSC(P)$ are not formed pointwise, in general. The infimum
in $\LSC(P)$ of a family of functions $f_i$ is the lower semicontinuous
envelope of the pointwise infimum. The intrinsic lower topology of the
lattice $\LSC(P)$ need no longer be the subspace topology induced by the
lower topology on $\oRp^P$; it can be strictly finer.

We now suppose that $P$ is a predomain.
By definition, a subbasis for the closed sets for the lower topology
in $\LSC(P)$ is given by the sets $\ua h=\{f\in \LSC(P)\mid h\leq f\}$
where $h$ ranges over $\LSC(P)$. If $f\not\in \ua h$, there is an $x_0\in
P$ such that $f(x_0)<h(x_0)$. Choose $r$ such that $f(x_0)<r<h(x_0)$. By lower
semicontinuity, there is a $y\llcurly x_0$ such that $r<h(y)$. Thus
$f \in W_{y,r)} =\{g\in \LSC(P)\mid g(x)< r \mbox{ for some } x\in\dua
y\}$ and $W_{y,r)}$ is disjoint from $\da h$. Moreover $W_{y,r}$ is
open for the lower topology, since it is the complement of the
subbasic lower closed set of all $f\in \LSC(P)$ below the simple
lower semicontinuous function $r\chi_{\dua y}$ which has value $r$ if $x\in
\dua u$ and value $0$ else. Thus the sets $W_{y,r}$ form a subbasis
for the lower topology of $\LSC(P)$.
 
\begin{lemma}\label{lem:env}
For every monotone map $g$ from a predomain $P$ to $\oRp$, the
lower semicontinuous envelope\footnote{The lower semicontinuous
  envelope as given by formula (Env) is standard in analysis. The
  formula given in the special situation of this lemma is standard in
  Domain Theory (see, e.g., \cite[]{compend}, cite[]{dom}). It has
  been rediscovered in \cite[Lemma 4.7]{ERS}, \cite[Lemma 2.2.1]{R}.}
is given by 
\[env(g)(x) = \sup_{y\ll x} g(y)\] for all $x\in P$
and the map $env\colon \MON(P)\to \LSC(P)$ preserves not only arbitrary
infima but also arbitrary suprema. 
\end{lemma}

\begin{proof}
In a predomain, an element $x$ has a neighborhood basis of
principal filters $\ua y$ with $y\llcurly x$. If $g$ is monotone, we
have that $\inf_{z\in\ua y}f(z)=f(y)$ and the above formula for the
lower semicontinuous envelope simplifies to $env(g)(x) =
\sup_{y\llcurly x} f(y)$.  

We now take a family of monotone functions $g_i\colon
P\to\oRp$ and we show that $env(\sup_i g_i) = \sup_i
env(g_i)$. Using the formula for $env(g)$ just proved we have indeed,
$env(\sup_i g_i)(x) = \sup_{y\llcurly x}\sup_i 
g_i(y) = \sup_i \sup_{y\llcurly x} g_i(x) =\sup_i env(g_i)(x)$.
\end{proof}

Since the map $env$ maps preserves arbitrary infima and arbitrary
suprema, it is continuous for the respective lower, upper and
interval topologies. It also has a lower adjoint $\alpha$ characterized by
\[g\leq \alpha(f) \mbox{ if and only if } env(g) \leq f \eqno{\rm (Adj2)}\]
for $f\in \LSC(P)$ and $g\in \MON(P)$. Explicitly, $\alpha(f) = \sup
\{g\in \MON(P)\mid env(g)\leq f\}$. 
  
We now finish the proof of claim (b) by considering a net $(f_i)_i$
in $\LSC(P)$. Suppose firstly that the net $f_i$ converges to some
$f\in \LSC(P)$ for the lower topology. Since $\alpha$ is a lower
adjoint, it is continuous for the lower topologies so that the net
$\alpha(f_i)$ converges to $\alpha(f)$ for the lower topology in
$\MON(P)$. This means that $\limsup_i \alpha(f_i)(x) \leq \alpha(f)(x)$
for every $x\in P$ by condition (Cond). Passing to the lower
semicontinuous envelope on both sides yields $sup_{y\llcurly
  x}\limsup_i \alpha(f_i)(y) \leq f(x)$ hence $\limsup_i
\alpha(f_i)(y) \leq f(x)$ whenever $y\llcurly x$ as in claim
(b). Suppose conversely that the latter property holds. In order to
prove that the net $(f_i)_i$ converges to $f$ we take any subbasic
neighborhood $W_{y,r}$ of $f$. Then $f$ satisfies $f(x_0) <r$ for some
$x_0\in \dua y$. Choose any $x$ such that $y\llcurly x\llcurly
x_0$. Since $\limsup_if_i(z) \leq f(x_0) <r$, there is an 
index $j$ such that $f_i(z) < r$ for all $i\geq j$ and we conclude
that $f_i\in W_{y,r}$ for all $i\geq j$. 

In order to prove claim (c) we first observe that $\LSC(P)$ is compact
for the interval topology, since $\MON(P)$ is compact for the
interval topology and the map $env\colon \MON(P)\to \LSC(P)$ is
continuous. The following lemma shows that the order in $\LSC(P)$ is
closed for the interval topology so that $\LSC(P)$ is a compact ordered
space.

\begin{lemma}\label{lem:sepa}
Let $f\not\leq h$ in $\LSC(P)$. Then there is a subbasic upper open
neighborhood $V$ of $f$ disjoint from some subbasic lower open
neighborhood $W$ of $h$.
\end{lemma}

\begin{proof}
Since $f\not\leq h$. There is an $x_0$ such that $f(x_0)>h(x_0)$. Choose an
$r$ with  $f(x_0) > r > h(x_0)$. By lower semicontinuity, there is a
$y\llcurly x_0$ such that $f(y)>r$. Now let $V_{y,r}$ be the sets of
all $f\in \LSC(P)$ such that $f(y)> r$ and $W_{y,r}$ the set of all
$f\in \LSC(P)$ such that $f(x)<r$ for some $x$ with $y\llcurly x$. Then
$V_{y,r}$ and $W_{y,r}$ are disjoint subbasic open sets for the 
upper and lower topology, respectively, containing $f$ and $h$,
respectively.       
\end{proof}



On $\oRp$ addition is jointly continuous with respect to each of
the three topologies (upper, lower and interval topology) as a map
$\oRp\times \oRp\to\oRp$. Multiplication is jointly continuous as a map
$\R_{>0}\times\oRp\to \oRp$ for these three topologies\footnote{There is no
way to extend the multiplication to all of $\oRp$ in such a way that it
remains continuous for the interval topology.This fact had
  been overlooked in \cite{ERS} and had led to misleading statements
  in \cite{ERS}. If we extend
multiplication by $+\infty\cdot 0=0 = 0\cdot(+\infty)$, it remains
continuous for the upper topology, if we extend it by  $+\infty\cdot
0=+\infty = 0\cdot(+\infty)$, it remains continuous for the lower
topology.}.
Thus $\oRp$ is a topological cone with respect to all of the three
topologies (upper, lower and interval topology), where a
\emph{topological cone} is a cone $C$ with a topology such that addition
and scalar multiplication are jointly continuous as maps $C\times
C\to C$ and $\R_{>0}\times C\to C$, respectively. This definition has
to be read with caution: The question which topology to use on
$\R_{>0}$; one has to use the upper, lower and interval topology,
respectively, in agreement with the topology used on $C$.  

Since $\oRp$ is a topological cone, the power $\oRp^P$ is a
topological cone, too, for the
pointwise defined addition and multiplication with real numbers
$r>0$, and this for each of the three topologies (lower, upper and
interval topology). For a preordered set $P$, the monotone functions
form a subcone $\MON(P)$.  
For a topological space $P$, the sum $f+g$ of two lower semicontinuous
functions $f,g\in \LSC(P)$ 
and the scalar multiple $rf$ for $0<r<+\infty$ are lower
semicontinuous, too. Thus $\LSC(P)$ is a subcone of
$\MON(P)$. Furthermore, if $P$ is a predomain, the map
$env\colon \MON(P)\to \LSC(P)$ is linear. This is easily verified using
the formula for the lower semicontinuous envelope in Lemma
\ref{lem:env}; but there is also a general argument that we present
after the statement of the next proposition. We conclude: 

\begin{proposition}\label{prop:funccones}
Let $P$ be a predomain. $\MON(P)$ and $\LSC(P)$ are ordered topological cones
for their intrinsic upper, lower and interval topologies, respectively. The
map $env\colon \MON(P)\to \LSC(P)$ is linear, monotone and continuous
for each of the three topologies.
\end{proposition}

Recall that, in the previous proposition, according to our definition
of a topological cone, the 
set $\R_{>0}$ of positive scalars has to be equipped with the
respective upper, lower, and interval topology. 

We will use the following observation several times:

\begin{observation}\label{obs}
Let $C$ and $D$ be cones each with a topology that agrees with the
upper topology on the rays $\R_{>0}\cdot a$. Then
every continuous monoid homomorphism $f\colon C\to D$ is homogeneous,
hence linear.  
{\rm  Indeed, by additivity one
obtains $f(qa) =qf(a)$ for every rational number $q>0$. For a real
number $r>0$ choose an increasing sequence $q_n$ of rational numbers
with supremum $r$. Then $a = \sup q_n a$ and $rf(a) = \sup_n q_nf(a)$
since $r\mapsto rx$ is supposed to be continuous for the respective
upper topologies.   Since $f$ is continuous for the respective upper
topologies, we finally obtain $f(ra) =f(\sup_nq_n a) = \sup_nf(q_na) =
\sup_n q_nf(a) = rf(a)$.} 
\end{observation}

One may ask, why we restrict scalar multiplication to $\R_{>0}$ and
why we do not extend it to $r=0$ and $r=+\infty$. The reason is that we
have to treat the three cases differently concerning such an
extension. While there is no continuous extension of scalar
multiplication to $\oRp$ for the interval topology, we can proceed as
follows for the two other cases. 

Using the upper topology, we may define $0\cdot r = 0
= r\cdot 0$ for all $r\in\oRp$ (including $r=+\infty$) and
$r\cdot(+\infty) = +\infty =(+\infty)\cdot r$ for $r>0$. This
multiplication is continuous on $\oRp$ for the upper topology and can
be extended pointwise to a multiplication of functions $g\in \MON(P)$ and $f\in
\LSC(P)$ with scalars $r\in \oRp$ which remains continuous for the
upper topologies and which satisfies all defining laws of scalar
multiplication in cones. 

Using the lower topology, we may define $0\cdot r = 0
= r\cdot 0$ for all $r<+\infty$ and
$r\cdot(+\infty) = +\infty =(+\infty)\cdot r$ for all $r\in \oRp$
(including $r=+\infty$. This
multiplication is continuous on $\oRp$ for the lower topology and can
be extended pointwise to a multiplication of functions $g\in \MON(P)$ and $f\in
\LSC(P)$ with scalars $r\in \oRp$ which remains continuous for the
lower topologies and which satisfies all defining laws of scalar
multiplication in cones.

\begin{remark}{\rm
In domain theory one usually stresses the Scott topology. In the
context of the this section, the Scott topology agrees with the upper
topology $\tau_{up}$. This is the case for $\oRp^P$, $\MON(P)$ and, in
case of a predomain $P$, also for $\LSC(P)$. The same holds for the
dual Scott topology and the lower topology   $\tau_{lo}$ in all of
these cases. The reason is that this phenomenon occurs in complete
completely distributive lattices in general (see, e,g,, \cite[Section
VII-3]{compend}). We have preferred to use the lower and upper topology
since their definition is simpler.}
\end{remark}

\subsection{Compact ordered and stably compact spaces}

Let us point out that in the cases under consideration each one of the
three topologies (upper, lower and interval topology) determines the
other two uniquely.

According to L. Nachbin \cite{N}, a compact space $(X,\tau)$ endowed with a
partial order $\leq$ the graph $G_\leq = \{(x,y)\mid
x\leq y\}$ of which is closed in $X\times X$ is called a \emph{compact ordered
space}. Such a space is always Hausdorff, since the diagonal in
$X\times X$ is closed.

To any compact ordered space $(X,\tau,\leq)$ we associate two other
topologies, the lower topology $\tau^{lo}$ and the upper topology
$\tau^{up}$. The closed sets of the upper (resp., lower) topology are the
$\tau$-open upper (resp., lower) sets. Thus, the open sets of the
upper (resp., lower) topology are the $\tau$-open upper (resp., lower)
sets. We will use the following characterization of these two derived
topologies: 

\begin{lemma}\label{lem:compord}
Let $(X,\tau,\leq)$ be a compact ordered space. Suppose that $\tau_1$
(resp., $\tau_2$) are topologies on $X$ that consists of $\tau$-open
upper (resp., lower) sets which are 
separating in the following sense: Whenever $x\not\leq y$, there are disjoint
sets $U\in \tau_1$ and $V\in \tau_2$ such that $x\in U$ and $y\in
V$. Then $\tau_1$ is the upper and $\tau_2$
the lower topology.  
\end{lemma}

\begin{proof}
Let $W$ be an arbitrary $\tau$-open upper set. We have to show that
$W$ belongs to $\tau_1$. For this choose any $x\in W$. It suffices to
show that there is a $U\in \tau_1$ such that $x\in U\subseteq W$
(since then $W$ is the union of open sets belonging to $\tau_1$). Thus
take any $y\not\in U$. Then $x\not\leq y$ and we can find disjoint
sets $U_y\in\tau_1$ and $V_y\in\tau_2$ such that $x\in U_y$ and $y\in
V_y$. The open sets $V_y$ cover the complement of $W$ which is a
closed hence compact set. Thus, finitely many of the $V_y$ cover the
complement of $W$. Take the intersection 
$U$ of the corresponding finitely many $U_y$. Then $x\in U\subseteq W$
and $U\in \tau_1$.
\end{proof}

There is an equivalent way to look at this situation. A topological
space $(X,\omega)$ is called \emph{stably compact} if it is compact,
locally compact, sober and coherent. By \emph{coherent} we mean that
the intersection of any two compact saturated subsets is compact. 

The relation between stably compact spaces and compact ordered spaces
is the following (see, e.g., \cite{AJK} or \cite[Section VI-7]{compend}):

To every stably compact space $(X,\omega)$ we associate a compact
ordered space $(X,\omega^p,\leq_\omega)$ in the following way:
$\leq_\omega$ is the specialization order associated with the topology
$\omega$. The topology $\omega^p$ is the coarsest refinement of the
given topology $\omega$ and the associated co-compact topology
$\omega^{cc}$ the closed sets of which are the $\omega$-compact saturated
subsets of $X$. Moreover, the original topology $\omega$ is the upper
topology associated with $\omega^p$ and the co-compact topology
$\omega^{cc}$ is the lower topology.

Conversely, Let $(X,\tau,\leq)$ be a compact ordered space. Then the
upper topology $\tau^{up}$ is stably compact. Its associated
co-compact topology is the lower topology and $\tau$ is the coarsest
common refinement of the associated upper and lower topologies. The
order $\leq$ agrees with the specialization order associated with the
upper topology.

This setting allows an alternative proof of Proposition \ref{prop:funcspaces}.
We use:

\begin{lemma}\label{lem:retract}
If $X$ is a stably compact space and $Y$ a retract of $X$, that is, if
there are continuous maps $\rho\colon X\to Y$ and $i\colon Y\to X$
such that $\rho\circ i$ is the identity in $Y$, then $Y$ is stably
compact, too.  
\end{lemma}

We now let $P$ be a predomain. We recall that
$(\MON(P),\tau,\leq)$ is a compact ordered space. Thus, 
its upper topology $\tau^{up}$ is stably compact. For its intrinsic
upper topology, $\LSC(P)$ is a subspace of $\MON(P)$ and even a retract
under the map $env\colon \MON(P)\to \LSC(P)$ which is continuous for the
upper topologies, since $env$ preserves arbitrary suprema by Lemma
\ref{lem:env}. Thus $\LSC(P)$ is stably compact for its intrinsic upper
topology $\tau^{up}$ by Lemma \ref{lem:retract}. By Lemma \ref{lem:compord} and
Lemma \ref{lem:sepa}, the intrinsic lower topology on $\LSC(P)$ agrees with the
co-compact topology $(\tau^{up})^{cc}$ and, hence,  
the compact Hausdorff topology $(\tau^{up})^p$ agrees with the
intrinsic interval topology of $\LSC(P)$. We summarize:

\begin{proposition}\label{prop:stablycomp}
Let $P$ be a predomain. Then $\LSC(P)$ is stably compact for its upper
topology. The associated co-compact topology is the lower topology and
the associated patch topology is the interval topology.
\end{proposition}

\subsection{The dual $M^*$ of a preCuntz semigroup}

For a predomain monoid $(M,+,0,\ll)$ its \emph{dual} $M^*$ is defined
to be the set of all 
lower semicontinuous monoid homomorphisms $\varphi\colon M\to \oRp$. 
Since the sum $\varphi+\psi$ of monoid homomorphisms $\varphi$ and
$\psi$ and also the scalar multiple $r\varphi$, $0<r<*\infty$ are
monoid homomorphisms, $M^*$ is a subcone of $\LSC(M)$. Since the
pointwise supremum of 
a directed family of lower semicontinuous monoid homomorphisms is again
not only lower semicontinuous but also a monoid homomorphism, $M^*$ is a
dcpo-monoid. But $M^*$ is not a domain. Let us investigate its
topological structure. 

%

As in Section \ref{subsec:3.2} we will use the set $M'$ of all monotone
monoid homomorphisms $\gamma\colon P\to\oRp$. Clearly, $M'$ is a
subcone of the cone $\MON(P)$ of all monotone maps from $P$ to $\oRp$. 

The central observation is:

\begin{lemma}\label{lem:env1}
For a preCuntz semigroup $M$, the lower semicontinuous envelope
$env(\gamma)$ of a monotone monoid 
homomorphism $\gamma\colon M\to\oRp$ is also a monoid homomorphism.
\end{lemma}

\begin{proof}
Given a monotone monoid homomorphism $\gamma$, recall that
$env(\gamma)(x) = \sup_{x'\llcurly x}\gamma(x')$. Thus, clearly
$env(\gamma)(0) = 0$. In order to show additivity, let $x,y\in
M$. Then $env(\gamma)(x) + env(\gamma)(y) =  \sup_{x'\llcurly
  x}\gamma(x') + \sup_{y'\llcurly y}\gamma(y') = \sup_{x'\llcurly
  x,y'\llcurly y} \gamma(x') +\gamma(y') = \sup_{x'\llcurly
  x,y'\llcurly y} \gamma(x'+y') \leq \sup_{z\llcurly x+y}\gamma(z) =
env(\gamma)(x+y)$, where we have used that the relation $\llcurly$ is
additive in $M$ for the inequality in the chain of equalities
above. The reverse inequality follows from the continuity of addition
in $M$ which implies that, if $z\llcurly x+ y$, then there are
$x'\llcurly x$ and $y'\llcurly y$ such that $z\leq x'+y'$. Thus
$\sup_{x'\llcurly x,y'\llcurly y} \gamma(x'+y') \geq \sup_{z\llcurly
  x+y}\gamma(z)$.  
This allows to read the above chain of equalities in the
with $\leq$ replaced by $\geq$.
\end{proof}

\newarrow{TeXonto} ----{->>}
\newarrow{Into} C---> 

Thus $env$ maps $M'$ onto $M^*$ and we have the following situation
where all the arrows denote linear maps:
\begin{diagram}
\oRp^M&\lInto&\MON(M) & \lInto &M'\\
               &&\dTeXonto<{env}\uInto&       &\dTeXonto{env}\uInto\\
               &&\LSC(M) & \lInto &{M^*}        
\end{diagram}  

We consider the restrictions to $M^*$ of our three topologies
on $\LSC(M)$:  
\begin{verse}
The \emph{weak$^*$upper topology} $\tau_{up}^*$, the restriction of the upper
topology on $\LSC(P)$. It is the weakest topology for which all the
point evaluations  
$\delta_x\colon \varphi\mapsto\varphi(x)\colon M^*\to\oRp$ are
lower semicontinuous, 

the restriction $\tau^*_{lo}$ to $M^*$ of the lower topology
$\tau_{lo}$ on $\LSC(M)$,

the restriction $\tau^*_{iv}$ to $M^*$ of the interval topology $\tau_{iv}$.
\end{verse}
We now are ready for our main result:

\begin{theorem}\label{th:dualcone}
Let $M$ be a preCuntz semigroup and $M^*$ its dual cone.

(a) For the topology $\tau^*_{iv}$ and the pointwise order $\leq$,
$M^*$ is a compact ordered topological cone.

(b) For the weak$^*$upper topology $\tau^*_{up}$, and similarly for the
topology $\tau^*_{lo}$, $M^*$ is a stably compact topological cone.
\end{theorem}

\begin{proof}
Lt us show that $M'$ is closed in $\MON(M)$ for the interval
topology. Convergence for the interval topology in $\MON(P)$ is
pointwise convergence in $\oRp$. Thus if $\gamma_i$ is a net in $M'$
that converges to some $\gamma\in \MON(M)$, then for $x,y\in M$,
$\gamma_i(x)$ converges to $\gamma(x)$, $\gamma_i(y)$ converges to
$\gamma(y)$  and $\gamma_i(x+y)$ converges to $\gamma(x+y)$. At the
other hand, $\gamma_i(x+y)= \gamma_i(x)+\gamma_i(y)$ converges to
$\gamma(x)+\gamma(y)$ by the continuity of addition on $\oRp$. Thus
$\gamma(x)+\gamma(y)=\gamma(x+y)$.

As a closed subcone of $\MON(M)$, $M'$ is a compact ordered cone for
the interval topology. Forming the lower semicontinuous envelope maps
$M'$ onto $M^*$ by Lemma \ref{lem:env1}. By Proposition
\ref{prop:funccones}, the map $env$ is continuous for the respective
interval topologies. Hence, $M^*$ is also compact for the topology
$\tau^*_{iv}$ hence a compact ordered space, and closed
in $\LSC(M)$ for the interval topology. We infer that
$(M^*,\tau_{iv}^*)$ is a compact ordered topological cone. 

$M^*$ is also a topological cone for the weak$^*$upper topology
$\tau_{up}^*$ and the topology $\tau_{lo}^*$ which are stably compact       
according to Proposition \ref{prop:stablycomp}, being the topologies
of open upper and lower sets, respectively, for the topology
$\tau_{iv}^*$. 
\end{proof}
From Proposition \ref{prop:funcspaces} we also deduce:

(a) A subbasis for the weak$^*$upper topology $\tau_{up}^*$ of
$M^*$ is given by:   
\[V_{x,r} = \{f\in M^*\mid f(x) > r\},\ x\in M, r\in\Rp\]
 A subbasis for the topology $\tau_{lo}^*$ by:
\[W_{y,r} = \{f\in M^*\mid f(x) < r \mbox{ for some } x\in \dua
y\},\ y\in M, r\in \Rp.\]
Together these subbases constitute a subbasis for the topology $\tau_{iv}^*$.

(b) A net $(f_i)_i$ of functions in $M^*$ converges to $f$  
for the weak$^*$upper upper topology $\tau_{up}^*$ if and only if: 
\[f(x) \leq \liminf_i f_i(x) \mbox{ for all } x\in M \eqno{\rm (upConv)}\]  
for the topology $\tau_{lo}^*$  if and only if: 
\[\limsup_i f_i(y) \leq f(x)\ \mbox{ whenever } y\llcurly x\mbox{ in } M
\eqno{\rm (loConv)}\] 
for the topology $\tau_{iv}^*$ if and only
conditions {\rm (upConv)} and {\rm (upConv)} hold.  

These result hold in particular for the dual of our basic example, the
preCuntz semigroup $C_0(X)_+$. 

\begin{remark}{\rm
The main proof technique for the results in this subsection consists
in considering first the cone $M'$ of order preserving linear functionals
$\lambda\colon M\to\oRp$; for those the compactness properties follow
from the Tychonoff Theorem on the compactness of product
spaces. Taking the lower semicontinuous envelope yields a continuous retraction
on the the lower semicontinuous monoid homomorphisms. This technique
has first been applied by Jung \cite{AJK} and is heavily used in
\cite{Plo,alaoglu}. In \cite{ERS} it is mentioned that in the proof of
Theorem 3.7 on the compactness of the space of traces the same idea
has been communicated to the authors by E. Kirchberg. In \cite[Theorem
4.8]{ERS} claim (a) of  
Theorem \ref{th:dualcone} has been proved for Cuntz semigroups.}
\end{remark}

\subsection{The bidual $M^{**}$}

Let $M$ be a preCuntz semigroup and $M^*$ its dual. By the universal
property of the round ideal completion (see \ref{prop:universal}), the
dual  $\cRI(M)^*$ of $\cRI(M)$ is canonically isomorphic (algebraically
and topologically) to the dual $M^*$ of $M$ (and also to the dual of
the round $\omega$-ideal completion of $M$ if $M$ is first countable.   

We may form the bidual $M^{**}$, the cone of all linear functionals 
$\Lambda\colon M^*\to\oRp$ that are lower semicontinuous with respect
to the weak$^*$upper topology $\tau^*_{up}$; this is equivalent to requiring
that these maps are monotone and lower semicontinuous with respect
to the patch topology $\tau^*_p$; indeed, by Proposition
\ref{prop:universal1} the patch open upper sets agree with the
weak$^*$upper open sets. We endow $M^{**}$ with the
pointwise order, addition and multiplication by scalars $r>0$. We note 
that $M^{**}$ is directed complete (under pointwise suprema). There is a 
natural map from $M$ into its bidual $M^{**}$: to very $x\in M$
we assign the point evaluation $\widehat x\colon \varphi\mapsto
\varphi(x)$. This map from $M$ to $M^{**}$ clearly is a monoid
homomorphism,  linear and monotone. We would like this map to be an
order embedding, that is, $x\not\leq y$ in $M$ implies   $\widehat
x\not\leq \widehat y$. For this it suffices to have the following
separation property:


\begin{separation}
Whenever $x\not\leq y$ in $M$, there is a $\varphi\in M^*$ such that
$\varphi(x)> \varphi(y)$. 
\end{separation} 

This separation property will not be true for Cuntz semigroups in
general. We provide a proof under the hypothesis that 
$M$ is a preCuntz cone: 

\begin{lemma}\label{lem:sep}
Whenever $x\not\leq y$ in a preCuntz cone $M$, there is a $\varphi\in
M^*$ such that $\varphi(x)> \varphi(y)$. 
\end{lemma}

\begin{proof}
Consider elements $x\not\leq y$. Then $\dda x\not\subseteq \dda y$,
that is, there is an element $z\ll x$ with
$z\not\ll y$. By interpolation we find an element $z'$ with $z\ll
z'\ll x$. Then $\dda z'\not\subseteq \dda y$, that is, $z'\not\leq y$.
Using interpolation we recursively find a sequence $x\gg x_1\gg x_2\gg
\dots \gg z'$. The set $U$ of all $u\in M$ such that $u\gg x_n$ for
some $n$ is a $\tau_\ll$-open neighborhood of $x$ contained in $\dua
z'$ whence 
$y\not\in U$. Moreover, $U$ is convex. Indeed, for elements $u,v\in U$
there is an $n$ such that $u,v\gg x_n$. It follows for every $r$ in
the open unit interval, $ru+(1-r)v\gg rx_n + (1-r)x_n = x_n$, that is
$ru+(1-r)v\in U$. We now can apply \cite[Corollary 9.2]{K} which tells
us that for every open convex set U in a semitopological cone and
every element $y$ not contained in U, there is a lower semicontinuous
linear functional $\varphi$ such that $\varphi(y)<1$ but
$\varphi(u)>1$ for all $u\in U$, in particular, $\varphi(x)>\varphi(y)$. 
\end{proof}

For every round ideal $I$ of $M$, let $\widehat I = \sup \{\widehat
x\mid x\in I\}$. Clearly, $\widehat I\in M^{**}$. Thus, we obtain a
map from $\cRI(M)$ to $M^{**}$ which is Scott-continuous.  Moreover
this map preserves $\ll$: 

\begin{lemma}
For round ideals $I$ and $J$ in a Cuntz cone, $I\ll J$ implies
$\widehat I \ll \widehat J$. 
\end{lemma}

\begin{proof}
We first consider elements $x\llcurly y$ in $M$. As $y =\sup_{r<1}ry$, there
is an $r<1$ such that $x\llcurly ry$. Let $U_x$ be the set of all
$\varphi\in M^*$ such that $\varphi(x)>1$, and similarly for
$U_{ry}$ and $U_y$. By definition, $U_x$, $U_{ry}$ and $U_y$ are
weak$^*$upper open and  
$U_x\subseteq U_{ry}\subset U_y$. We want to show that there is a
compact saturated set $K$ such that $U_x\subseteq K\subseteq
U_y$. Indeed, let  $\varphi_i$ be a net in $U_x$ converging to some
$\varphi$ for the topology $\tau_{lo}^*$. Let us show that $\varphi\in
U_y$. Indeed, $\varphi(ry) \geq \limsup_i\varphi_i(z)$ for all $z\ll
ry$, in particular, $\varphi(ry) \geq limsup_i \varphi_i(x) \geq
1$. Thus $r\varphi(y)\geq 1$ whence $\varphi(y) =\frac{1}{r} 
>1$, that is $\varphi\in U_y)$. From this we conclude that $\widehat x \ll
\widehat y$ in $M^{**}$. The claim for ideals is a direct consequence,
for if $I\ll J$ there are elements $x\llcurly y$ in $J$ such that
$I\subseteq \dda x$.  
\end{proof}

The following question arises:

\begin{question}\label{qu:bidual}
If $M$ is a preCuntz cone, is $M^{**}$ isomorphic to the round
ideal completion $\cRI(M)$? More precisely, given any $\Lambda\in M^{**}$,
is there a round ideal $J$ in $M$ such that $\Lambda=\widehat J$. 
\end{question}

The answer to this question is 'yes' in the case of our basic example, the cone
$C_0(X)_+$ for a locally compact Hausdorff space $X$: In this case,
$(C_0(X)_+)^{**}$ is naturally isomorphic to
the cone $\LSC(X)$ of all lower semicontinuous functions 
 $f\colon X\to\oRp$, which is the round ideal completion of
 $C_0(X)_+$ according. 
 Indeed, $(C_0(X)_+)^*$ corresponds to the cone of all continuous
 valuations (a topological variant of measures) on $X$ and the 
 claim is a special case of the Schr\"oder-Simpson Theorem (see
 \cite[Theorem 2.15]{K1} or \cite{GL} for a short proof).

In the search for an affirmatively answer to the question above for a
preCuntz cone $M$, one can use 
\cite[Corollary 4.5]{K1} which tells us that every lower semicontinuous linear 
functional $\Lambda$ on $M^*$ is the pointwise supremum of functionals
of the form $\widehat x_i$, $x_i\in M$. If we can show that we can
choose this set of $x_i$ 
to be directed, then we have a positive answer to our
question. Indeed, in this case 
the $y\llcurly x_i$ for some $i$ form a round ideal $J$ of $M$ such
that $\widehat J =\Lambda$. 

Robert \cite{R1} has investigated the relation between $M$ and the
double dual $M^{**}$ for Cuntz semigroups that are not already
cones. Here the problem is to embed $M$ into a cone  
which he succeeds by a kind of tensor product construction but under
additional hypotheses on the Cuntz semigroup. 
%
%
%
%
%
%



\section{Traces on C$^*$-algebras}

We now turn to C$^*$-algebras.
Let $A$ be a $C^*$-algebra. The elements of the form $a=xx^*$, $x\in
A$, are called \emph{positive}. These elements form a cone denoted $A_+$. On
$A_+$ we use the topology induced by the norm of the C$^*$-algebra and
the natural order $a\leq b$ if $b-a\in A_+$. We
refer to standard references for background material. 

A \emph{lower semicontinuous trace} is a lower semicontinuous monoid
homomorphism $t\colon A_+\to\oRp$ such that $t(xx^*) = t(x^*x)$ for
all $x\in A$. We denote by $T(A)$ the set of all traces. $T(A)$
becomes an ordered cone for the
pointwise defined order, addition and multiplication by real numbers
$r> 0$ . We would like to view $A_+$ as a predomain in such a way that
$T(A)$ is its dual. We let us guide by the basic example \ref{ex:basic}
$C_0(X)_+$ \ref{ex:basic}.

We remark that a lower semicontinuous trace satisfies $t(ra) =rt(a)$
for $r\in Rp$ and $a\in A_+$ and, hence, is a linear map on
$A_+$. This follows from the properties of being a monoid homomorphism
and lower semicontinuity.

\subsection{$A_+$ as a preCuntz semigroup} 

Every element $a\in A_+$ generates a commutative $C^*$-subalgebra
$C^*(a)$ of $A$. By Gelfand's representation theorem, there is an
isometrical isomorphism $i_a\colon C^*(a)\to C_0(X)$ for some locally
compact Hausdorff space $X$. We denote by $(a-\varepsilon)_+$ the element of
$C^*(a)$ that corresponds to the function $(i_a(a)-\varepsilon)_+$
in $C_0(X)$.

As a first try we define $a\ll b$ for elements $a,b\in A_+$ if $a\leq
(b-\varepsilon)_+$ for some $\varepsilon > 0$. 
In this way $A_+$ becomes a predomain. We first
check the interpolation property. We have
indeed $0\ll b$ for every $b\in A_+$ and if $a_i\ll c$ for $i=1,2$,
then $a_i\leq  (b-\varepsilon)_+$ for some $\varepsilon>0$. For
$c=b-\frac{\varepsilon}{2}$ we then have 
$a_i\ll c \ll b$ for $i=1,2$. Transitivity follows from the
fact that $a\ll b$ implies 
$a\leq b$ and that $a\leq b \ll c$ implies $a\ll c$.

The relation $\ll$ just defined will not have the desired
properties. Following Cuntz and Pedersen \cite{CP}, one should take in
account an equivalence relation that identifies elements that are
identified by every lower semicontinuous trace.
Since traces identify the elements $xx^*$ and $x^*x$, we consider
$xx^*$ and $x^*x$ 
to be equivalent. For a sequence $(x_i)_i$ of elements in $A$, if the
sums $\sum_ix_ix_i^*$ and $\sum_ix_i^*x_i$ both converge, a lower
semicontinuous trace will also identify these two sums.   

 According to \cite{CP}, two elements $a$ and $a'$ in $A_+$
are \emph{Cuntz-Pedersen equivalent} and we write $a\sim a'$ if  there
is a sequence  $x_n$ in  $A$ such that  $a = \sum_n x_nx_n^*$ and $
a'=\sum_n x_n^*x_n $. 

The relation $\sim$ is indeed an equivalence relation (transitivity is
by no means straightforward). 
 Moreover, $\sim$ is countably additive, that is,  $a_n\sim b_n$
 implies $\sum_n a_n\sim \sum_n b_n$ provided that the respective
 infinite sums converge. We refer to \cite[Section 2]{CP} for proofs.
 Clearly, $\sim$ is a congruence relation, that is, for all
$a,a',b\in A_+$ and $r\in \Rp$ one has: 
\[a\sim a' \mbox{ implies } a+b\sim a'+b, ra\sim ra' \]

The \emph{Cuntz-Pedersen preorder} on $A_+$ is defined by: \[a
 \precsim b \mbox{ if there is an } a' \in A \mbox{ such that } a \sim
 a' \leq b  \eqno{\rm (CPP)}\] 
Note that $a\leq b$ implies $a\precsim b$. 

We want to replace the Cuntz-Pedersen preorder by a relation that we
like to call the \emph{Cuntz-Pedersen approximation relation}
$\llcurly$ defined as follows:
\[a \llcurly b \mbox{ if there is a }
  a'\in A_+ \mbox{ and an } \varepsilon >0 \mbox{ such
  that } a\sim a' \leq (b-\varepsilon)_+ \]
Equivalently:
\[a \llcurly b \mbox{ if there are an } \varepsilon >0 
 \mbox{ such
  that } a\precsim  (b - \varepsilon)_+ \] 
We note that in particular,
$(b-\varepsilon)_+ \llcurly (b - \frac{\varepsilon}{2})_+
\llcurly b$.  

\begin{proposition}\label{prop:aplus}
For every C$^*$-algebra $A$, $(A_+,+,0,\llcurly)$ is a first
countable  preCuntz semigroup.
\end{proposition} 

For the proof we first observe that $d\llcurly c$ implies $d\precsim
(c-\varepsilon)_+\leq c$ for some $\varepsilon >0$ whence $d\precsim c$.
We now show that $\llcurly$ endows $A_+$ with the structure of a
predomain.  

For transitivity, let $d\llcurly c \llcurly a$. Then $d \precsim c$ as
we just noticed and $c\precsim (a-\varepsilon)_+$. We infer
$d\precsim (a-\varepsilon)_+$ from the transitivity of $\precsim$ whence
$d\llcurly a$. 

For interpolation we notice that $0\llcurly a$
for every $a\in A_+$ so that we have
(Int0).  For (Int2), suppose that $c_i\llcurly a$ for $i=1,2$.
Then there is an $\varepsilon >0$ such that $c_i\precsim
(a-\varepsilon)_+$ for $i=1,2$. Since 
$(a-\varepsilon)_+ \llcurly (a-\frac{\varepsilon}{2})_+ \llcurly
a$, we may choose $c=(a-\frac{\varepsilon}{2})_+$ and we have
$c_i\llcurly c\llcurly a$ for $i=1,2$. 

It remains to show that addition preserves $\llcurly$ and is
continuous. For this we use a result by Elliott, Robert and Santiago
\cite[Proposition 2.3]{ERS}: Given $a,b\in A_+$ and
$\varepsilon >0$, there is a $\delta >0$ such that 
\begin{eqnarray}\label{eq:(1)}
(a-\varepsilon)_+ + (b-\varepsilon)_+ &\precsim& (a+b-\delta)_+\\ \label{eq:(2)}
(a+b -\varepsilon)_+ &\precsim& (a-\delta)_+ + (b-\delta)_+ 
\end{eqnarray}
Indeed, these two inequalities are equivalent to the following
properties which express the additivity and the continuity of the
relation $\llcurly$, respectively:
\begin{eqnarray}\label{eq:(3)}
a'\llcurly a, b' \llcurly b &\implies& a'+b'\llcurly a+b\\ 
\label{eq:(4)}
c\llcurly a+b &\implies& \exists a'\llcurly a, b'\llcurly b.\ \ c\llcurly a'+b'
\end{eqnarray}
This finishes the proof of Proposition \ref{prop:aplus}.\\

The natural preorder of the predomain $(A_+,\llcurly)$ according to
\ref{subsec:2.3} is defined by 
$a\precsim_{CP} b$ if $\dda a\subseteq \dda b$. More explicitly,
$a\precsim_{CP} b$  if for every 
$\varepsilon >0$ there is a $\delta>0$ such that
$(a-\varepsilon)_+\precsim (b-\delta)_+$. This preorder has already been
considered by Robert \cite{R}. Thus,
if $c'\precsim_{CP} c\llcurly a\precsim_{CP} a'$, then
$c'\llcurly a'$. 
From results due to Robert \cite{R} it follows that the converse is not
true, that is  $\dda a \subseteq \dda b$ does not imply $a\precsim
b$, in general.

It is natural to ask whether the natural preorder
$\precsim_{CP}$ agrees with the Cuntz-Pedersen preorder $\precsim$.
L.~Robert \cite[Proposition 2.1(iii)]{R} has proved the implication:
\begin{equation}\label{eq:CP}
a\precsim b \implies a\precsim_{CP}b
\end{equation}
But in the same paper, Robert \cite{R} exhibits an example that shows
that the converse does not hold, in general.

The proof for the implication (\ref{eq:CP}) is surprisingly
sophisticated. One refers to a lemma due to 
Kirchberg and R\o rdam \cite[Lemma 2.2]{KR}: 
If $\varepsilon > \norm{a - b}$ then there is a contraction
$d$ in $A$ such that $(a -\varepsilon)_+=dbd^*$. 
From this, one deduces \cite[Lemma 2.2]{ERS}: 
\begin{equation}\label{eq:2.2}
\norm{a-b}<\varepsilon \implies (a-\varepsilon)_+\precsim b
\end{equation}
One then shows the following refinement:
\begin{equation}\label{eq:2.2refined}
\norm{a-b}<\varepsilon \implies \exists \delta>0.\
(a-\varepsilon)_+\precsim (b-\delta)_+ 
\end{equation}
Suppose indeed $\norm{a-b}<\varepsilon$. Since $(b-\delta)_+$
converges (in norm) to $b$ for $\delta\to 0$, there is some $\delta>0$
such that still $\norm{a-(b-\delta)_+}<\varepsilon$. Now
(\ref{eq:2.2refined}) follows from  (\ref{eq:2.2}).

One further uses from \cite[Proposition 2.3]{ERS}  
\begin{eqnarray}\label{eq:11}
(xx^* -\varepsilon)_+ &\sim& (x^*x-\varepsilon)_+\end{eqnarray}
for every element $x$ of the C$^*$-algebra $A$ and $\varepsilon>0$. 
 From (\ref{eq:(2)}) (see also \cite[proof of
Proposition 2.1(i)]{R} we deduce:
\begin{eqnarray}
  \label{eq:(12)}
a \leq b &\implies& \forall \varepsilon>0.\ \exists \delta> 0.\
(a-\varepsilon)_+\precsim (b-\delta)_+
\end{eqnarray} 
Indeed, if $a\leq b$, then $b =
a+(b-a)$ and $b-a\in A_+$. Thus, for $\varepsilon >0$ we can find a
$\delta>0$ such that $(a-\varepsilon)_+ + ((b-a)-\varepsilon)_+ \precsim
(b-\delta)_+$. It follows that
$(a-\varepsilon)_+\precsim(b-\delta)_+$.

We are now ready for the proof of the implication (\ref{eq:CP}).
Suppose $a\precsim b$. There is a sequence $x_n$ of elements in $A$ such
that $a =\sum_{n=1}^\infty x_nx_n^*$ and $a'=\sum_{n=1}^\infty
x_n^*x_n\leq b$. Consider any $\varepsilon >0$. There is an $N$ such
that $\norm{a -\sum_{n=1}^N x_nx_n^*}<\varepsilon$. The following
chain of arguments shows that $a\precsim_{CP} b$:
\[\begin{array}{rcll}
(a-\varepsilon)_+&\precsim& (\sum_{n=1}^N x_nx_n^*-\delta)_+&
\mbox{for some } \delta>0 \mbox{ by (\ref{eq:2.2refined})}\\
  &\precsim& \sum_{n=1}^N (x_nx_n^* -\delta_1)_+&\mbox{for some }
  \delta_1>0 \mbox{ by (\ref{eq:(2)}) }\\ 
  &\sim& \sum_{n=1}^N(x_n^*x_n -\delta_1)_+ &\mbox{ by (\ref{eq:11})} \\
  &\precsim & (\sum_{n=1}^N x_n^*x_n - \delta_2)_+&\mbox{for some }
  \delta_2>0 \mbox{ by (\ref{eq:(1)}) }\\
  &\precsim& (b-\delta_3)_+& \mbox{for some } \delta_3>0 
   \mbox{ by (\ref{eq:(12)})} 
\end{array}\]

\subsection{The cone $T(A)$ of traces}

We are ready now to apply our results on the dual of a preCuntz
semigroup to the the preCuntz semigroup $(A_+,+,0,\llcurly)$ on the
positive cone of a C$^*$-algebra $A$. We first show that the cone
$T(A)$ of traces is the dual of the preCuntz semigroup
$(A_+,+,0,\llcurly)$:

\begin{lemma}
The lower semicontinuous traces on $A_+$ agree with the lower
semicontinuous monoid homomorphisms from the preCuntz semigroup
$(A_+,+,0,\llcurly)$ to $\oRp$. 
\end{lemma}

\begin{proof}
Consider a monoid homomorphism $t\colon A_+\to\oRp$ satisfying
$t(a)=t(a')$ whenever $(a\sim a'$. We want to show that $\lambda$ is
lower semicontinuous for the norm topology on $A_+$ if and only if it
is lower semicontinuous for the predomain structure $\llcurly$. Thus
let $r$ be a nonnegative real number and look at the set $U=\{a\in
A_+\mid t(a)>0\}$. We have to show that $U$ is open for the norm
topology if and only if it is open for the c-space topology
$\tau_{\llcurly}$ associated with $\llcurly$.  

Suppose first that $U$ is open for the norm topology and look at any
element $a\in U$. Since  
$(a-\varepsilon)_+)$ converges to $a$ with respect to the norm, when
$\varepsilon$ goes to $0$, we have $(a-\varepsilon)_+)\in U$ for
$\varepsilon$ small enough. The we have found an element
$b=(a-\varepsilon)_+)\in U$ such that $b\llcurly a$. We secondly look
at any element $c\in A_+$ with $a\llcurly c$. Then there is an $a'$
such that $a\sim a'\leq (c-\varepsilon)_+$ for some $\varepsilon
>)$. Then $t(a)=t(a')\leq t(c-\varepsilon)_+\leq t(a)$ since a monoid
homomorphism on $A_+$ preserves the order $\leq$. Hence $r<t(a)\leq
t(c)$, that is, $c\in U$. 
Thus, $U$ is open for the c-space topology $\tau_{\llcurly}$.  

Suppose conversely that $U$ is open for the c-space topology
$\tau_{\llcurly}$ and choose any $a\in U$. We want to show that there
is an $\varepsilon >0$ such that $b\in U$ for every $b$ such that
$\norm{a-b}<\varepsilon$. There is an $\varepsilon >0$ such that
$(a-\varepsilon)_+ \in U$. For every $b$ with $\norm{a-b}<\varepsilon$
there is a $\delta>0$ such that $(a-\varepsilon)_+\precsim
b$ by (\ref{eq:2.2}), whence  $(a-\varepsilon)_+\precsim_{CP}
b$ by the previous lemma. And the latter implies $b\in U$. 
\end{proof}

We now may apply Theorem \ref{th:dualcone} and we obtain the following
improvement of results by Elliott, Robert, Santiago \cite{ERS}:

\begin{corollary} Let $A$ be a C$^*$-algebra and $T(A)$ the cone of
lower semicontinuous traces.

(a) Equipped with the topology $\tau_{iv}^*$, $T(A)$ is an ordered
compact topological cone, that is, addition and scalar multiplication
are order preserving and jointly continuous,  where $\R_{>0}$ is
endowed with the usual Hausdorff topology.

(b) Equipped with the weak$^*$upper topology, $T(A)$ is a
stably compact topological cone, that
is, addition and scalar multiplication are continuous, where $\R_{>0}$ is
endowed with the upper topology.

(c) Equipped with the lower topology $\tau_{lo}^*$, $T(A)$ is a
stably compact topological cone, that
is, addition and scalar multiplication are continuous, where $\R_{>0}$ is
endowed with the lower topology.
\end{corollary}

Subbases and convergence for the three topologies involved in the
above corollary can be described as in the text following the proof of
Theorem \ref{th:dualcone}.

The dual $T(A)^*$ of the cone of traces consisting of the lower
semicontinuous linear functionals from $T(A)$ to $\oRp$ contains the
round ideal completion $\cRI(A_+)$ of $(A_+,+,0,\llcurly)$ as a
subcone via the map $J\mapsto \widehat J$, where  
$\widehat J\colon T(A)\to \oRp$ is defined by $\widehat J(\varphi)=
\sup_{x\in J}\varphi(x)$. This map is also an order embedding by
\ref{lem:sep}. Our general question \ref{qu:bidual} can be reformulated in
this special case:

\begin{question}
Is the the dual $T(A)^*$ equal to the round ideal completion
$\cRI(A_+)$ of $(A_+,+,0,\llcurly)$? More precisely, given any lower
semicontinuous linear map $Lambda\colon T(A)\to \oRp$, is there a
round ideal $J$ in $(A_+,+,0,\llcurly)$ such that $\Lambda =\widehat J$. 
\end{question} 

The answer to this question is 'yes' for commutative C$^*$-algebras as
we have indicated after \ref{qu:bidual}.

\end{document}